\documentclass[11pt]{amsart}
\usepackage[left=2.2cm,tmargin=2.5cm,bmargin=2.5cm,right=2.2cm]{geometry}

\usepackage{amsfonts}
\usepackage{amssymb}
\usepackage{hyperref}
\usepackage{stmaryrd}

\usepackage{tikz}
\usetikzlibrary{matrix,arrows,decorations.pathmorphing,cd}

\usepackage[capitalise]{cleveref}

\theoremstyle{plain}
\newtheorem{thm}{Theorem}[section]
\newtheorem{lem}[thm]{Lemma}
\newtheorem{cor}[thm]{Corollary}
\newtheorem{prop}[thm]{Proposition}
\newtheorem{conj}[thm]{Conjecture}
\newtheorem{ex}[thm]{Example}

\theoremstyle{definition}
\newtheorem{defi}[thm]{Definition}
      
\theoremstyle{remark}
\newtheorem{remark}[thm]{Remark}

\newcommand{\lemref}[1]{\hyperref[#1]{Lemma \ref*{#1}}}
\newcommand{\thmref}[1]{\hyperref[#1]{Theorem \ref*{#1}}}
\newcommand{\propref}[1]{\hyperref[#1]{Proposition \ref*{#1}}}
\newcommand{\corref}[1]{\hyperref[#1]{Corollary \ref*{#1}}}
\newcommand{\defref}[1]{\hyperref[#1]{Definition \ref*{#1}}}
\newcommand{\remref}[1]{\hyperref[#1]{Remark \ref*{#1}}}
\newcommand{\conjref}[1]{\hyperref[#1]{Conjecture \ref*{#1}}}

\def \F {\mathbb{F}}

\makeatletter
\newcommand*{\defeq}{\mathrel{\rlap{%
                     \raisebox{0.27ex}{$\m@th\cdot$}}%
                     \raisebox{-0.27ex}{$\m@th\cdot$}}%
                     =}

\numberwithin{equation}{section}
\makeatother

\makeatletter
\def\@setcopyright{}
\def\serieslogo@{}
\makeatother

\input xy
\xyoption{all} 

\title{The Tamely Ramified Geometric Quantitative Minimal Ramification Problem}

\author{Mark Shusterman}
\address{Department of Mathematics, Harvard University, 1 Oxford Street, Cambridge, MA 02138, USA}
\email{mshusterman@math.harvard.edu}

\begin{document}

\begin{abstract}

We prove a large finite field version of the Boston--Markin conjecture on counting Galois extensions of the rational function field with a given Galois group and the smallest possible number of ramified primes. 
Our proof involves a study of structure groups of (direct products of) racks.

\end{abstract}

\maketitle

\section{Introduction}

Let $G$ be a nontrivial finite group.
We denote by $d(G)$ the least cardinality of a generating set of $G$, and by $d_\lhd(G)$ the least cardinality of a subset of $G$ generating $G$ normally.
That is, $d_\lhd(G)$ is the smallest positive integer $d$ for which there exist $g_1, \dots, g_d \in G$ such that $G$ has no proper normal subgroup containing $g_1, \dots, g_d$. Equivalently, it is the least positive integer $d$ for which there exist $d$ conjugacy classes in $G$ that generate $G$. 
Let $G' = [G,G]$ be the commutator subgroup of $G$, and let $G^{\textup{ab}} = G/[G,G]$ be the abelianization of $G$. 
It is a standard group-theoretic fact, see for instance \cite[Theorem 10.2.6]{NSW}, that
\[
d_\lhd(G) = \max\{d(G^{\textup{ab}}), 1\} =
\begin{cases}
d(G^{\textup{ab}}) & G^{\textup{ab}} \neq \{1\}, \\
1 & G^{\textup{ab}} = \{1\}.
\end{cases}
\]

\begin{defi}

A $G$-extension of a field $L$ is a pair $(K, \varphi)$ where $K$ is a Galois extension of $L$ and $\varphi \colon \operatorname{Gal}(K/L) \to G$ is an isomorphism.
Abusing notation, we will at times denote a $G$-extension simply by $K$, tacitly identifying $\operatorname{Gal}(K/L)$ with $G$ via $\varphi$.

\end{defi}

Since $\mathbb Q$ has no unramified extensions, for every tamely ramified $G$-extension $K/\mathbb Q$, the inertia subgroups of $G$ generate it and are cyclic.
As the inertia subgroups of primes in $K$ lying over a given prime of $\mathbb Q$ are conjugate,
it follows that the number of primes of $\mathbb{Q}$ ramified in $K$ is at least $d_\lhd(G)$.

Going beyond the inverse Galois problem, \cite{BM} conjectures that the lower bound $d_\lhd(G)$ on the number of ramified primes in a tamely ramified $G$-extension of $\mathbb Q$ is optimal.

\begin{conj} \label{TamelyRamifiedBoston}

There exists a (totally real) tamely ramified $G$-extension $K/\mathbb{Q}$ such that the number of (finite) primes of $\mathbb Q$ ramified in $K$ is $d_\lhd(G)$.

\end{conj}

The assumption that $K$ is totally real can also be stated as $K/\mathbb Q$ being split completely at infinity, or as complex conjugation corresponding to the trivial element of $G$. \cite{BM} does not restrict to this case, allowing arbitrary ramification at infinity, but suggests that this case is perhaps more interesting (or more challenging).
\cite{BM} also verifies the conjecture for $G$ abelian.

Let $p$ be a prime number, and let $q$ be a power of $p$.
To study a problem analogous to \cref{TamelyRamifiedBoston} over the rational function field $\F_q(T)$ in place of $\mathbb Q$, we introduce two restrictions that can perhaps make the situation over $\F_q(T)$ more similar to the one over $\mathbb Q$. First, we consider only regular extensions $K/\F_q(T)$, namely those for which every element of $K$ that is algebraic over $\F_q$ lies in $\F_q$. In particular, this rules out constant extensions of $\F_q(T)$ - a family of everywhere unramified extensions that do not have an analog over $\mathbb Q$.
Second, we assume that $p$ does not divide $|G|$, so that every $G$-extension of $\F_q(T)$ is tamely ramified.
The following is a special case of conjectures made in \cite{DeW14, BSEF}.

\begin{conj} \label{TamelyRamifiedBostonFF}

Let $q$ be a prime power coprime to $|G|$.
Then there exists a regular $G$-extension $K/\F_q(T)$ (split completely at infinity), such that the number of primes of $\F_q(T)$ ramified in $K$ is $d_\lhd(G)$. 

\end{conj}

Giving such an extension of $\F_q(T)$ is equivalent to producing a morphism of smooth projective geometrically connected curves $g \colon Y \to \mathbb{P}^1$ over $\F_q$ having the following three properties.

\begin{itemize}

\item There exist monic irreducible polynomials $P_j \in \F_q[T]$ with $1 \leq j \leq d_{\lhd}(G)$ such that $g$ is \'etale away from the points of $\mathbb{P}^1$ corresponding to the (zeros of the) polynomials $P_j$, and $g$ is ramified at these points.
In other words, the set of roots of the polynomials $P_j$ is the branch locus of $g$.

\item We have $\operatorname{Aut}(g) \cong G$ and this group acts transitively on the geometric fibers of $g$.

\item The fiber of $g$ over $\infty$ contains an $\F_q$-point.

\end{itemize}

We refer to \cite{DeW14, BSEF, BSS} (and references therein) for some of the progress made on \cref{TamelyRamifiedBoston} and \cref{TamelyRamifiedBostonFF}.
For example, some results have been obtained in case $G$ is the symmetric group, or a dihedral group, assuming Schinzel's Hypothesis H on prime values of integral polynomials.

\cite{BM} goes on to propose, among other things, a quantitative version of their conjecture that we restate here.
For that, we fix conjugacy classes $C_1, \dots, C_{d_{\lhd}(G)}$ of $G$ that generate $G$, and put 
\[
C = \bigcup_{j=1}^{d_\lhd(G)} C_j.
\]
Furthermore, we assume that for each $1 \leq j \leq d_\lhd(G)$ and $g \in C_j$, every generator of the cyclic subgroup $\langle g \rangle$ lies in $C_j$.

For a number field $K$ put 
\[
\operatorname{ram}(K) = \{ p : p \ \text{is a prime number ramified in} \ K\}, \quad D_K = \prod_{p \in \operatorname{ram}(K)} p.
\]
Recall that $\operatorname{ram}(K)$ is the set of prime numbers dividing the discriminant of $K$, so $D_K$ is the radical of this discriminant.

For a positive integer $d$ and a positive real number $X$, 
we denote by $\Pi_d(X)$ the probability that a uniformly random positive squarefree integer less than $X$ has exactly $d$ (distinct) prime factors.

\begin{conj} \label{BMconjecture}

Let $X$ be a positive real number, and let $\mathcal E^C(G;X)$ be the family of (isomorphism classes of) totally real tamely ramified $G$-extensions $K$ of $\mathbb Q$ for which $D_K < X$ and some (equivalently, every) generator of each nontrivial inertia subgroup of $\operatorname{Gal}(K/\mathbb Q)$ lies in $C$.
Then there exists a positive real number $\delta^{G,C}$ such that as $X \to \infty$ we have
\begin{equation*}
\sum_{K \in \mathcal E^C(G;X)} \mathbf{1}_{|\operatorname{ram}(K)| = d_{\lhd}(G)} \sim  \delta^{G,C} \cdot \Pi_{d_{\lhd}(G)}(X) \cdot |\mathcal E^C(G;X)|.
\end{equation*}

\end{conj}

We note that the family $\mathcal{E}(G;X)$ of totally real tamely ramified $G$-extensions $K$ of $\mathbb Q$ with $D_K < X$
is finite by a classical result of Hermite.

\cref{BMconjecture} suggests that the chances of $D_K$ having $d_\lhd(G)$ prime factors asymptotically equal the chances of a squarefree number having $d_\lhd(G)$ prime factors. It is also possible to restate this heuristic in the following equivalent way. For $K \in \mathcal E^C(G;X)$ and $1 \leq j \leq d_\lhd(G)$, let $D_K(j)$ be the product of all the primes $p \in \operatorname{ram}(K)$ for which the generators of the inertia subgroups of $\operatorname{Gal}(K/\mathbb Q)$ for primes of $K$ lying over $p$ belong to $C_j$.
Then $D_K(j) \neq 1$ for every $1 \leq j \leq d_\lhd(G)$, and
\[
D_K = \prod_{j=1}^{d_\lhd(G)} D_K(j).
\]
\cref{BMconjecture} suggests that for every $1 \leq j \leq d_\lhd(G)$, the chances that $D_K(j)$ is a prime asymptotically equal the chances that a squarefree number is a prime, and that these events are asymptotically independent over $j$.

\begin{remark} \label{LocalizedRmk}

To elaborate on the last point, it is possible to give a localized version of \cref{BMconjecture} where we only sum over those $K$ in $\mathcal E^C(G;X)$ with $D_K(j)$ having order of magnitude $X_j$ for every $1 \leq j \leq d_\lhd(G)$ where $X_j \to \infty$ are real numbers whose product is $X$. In such a version we would replace $\Pi_{d_\lhd(G)}(X)$ with the product over $1 \leq j \leq d_\lhd(G)$ of the odds that a squarefree number with order of magnitude $X_j$ is a prime number, which is $\frac{\zeta(2)}{\log X_j}$ where $\zeta$ is the Riemann zeta function. We insist that each $X_j \to \infty$ (and not just their product) because if some $X_j$ does not grow, then (at least) one of the others is substantially more likely to be a prime as it is coprime to $X_j$ by definition.

\end{remark}

The reason we expect the appearance of an arithmetic correction factor $\delta^{G,C}$ is that, for a given prime number $r$, the probability that $r$ divides $D_K$ for a uniformly random $K \in \mathcal E^C(G;X)$ may not quite converge to $1/r$ (but perhaps to a different value) as $X \to \infty$. Heuristics of analytic number theory would then suggest that $\delta^{G,C}$ is a product over all the primes $r$ of certain expressions involving these limiting probabilities.

The odds that a uniformly random positive squarefree integer less than $X$ with exactly $d_\lhd(G)$ prime factors (which models $D_K$ for a random $K \in \mathcal E^C(G;X)$ with $|\operatorname{ram}(K)| = d_\lhd(G)$) is coprime to $|G|$ approach $1$ as $X \to \infty$.
This justifies (to some extent) our avoidance of wild ramification throughout, since even if we were to include wildly ramified extension in our count, there would be no apparent reason to change the conjectured asymptotic.

\begin{remark}

It is also possible to state a version of \cref{BMconjecture} for $\mathcal E(G;X)$ in place of $\mathcal E^C(G;X)$. Such a version (and its function field analog) may require more elaborate correction factors in place of $\delta^{G,C}$. We do not pursue this direction in the current work.

\end{remark}

Additional motivation for \cref{BMconjecture} comes from \cite{BE} suggesting (roughly speaking) that the Galois group over $\mathbb Q$ of the maximal extension of $\mathbb Q$ unramified away from a random set of primes of a given finite cardinality follows a certain distribution.
\cref{BMconjecture} can then be viewed as predicting the asymptotics of certain moments of such a distribution.

\cref{BMconjecture} remains open for every nonabelian $G$, see \cite[Section 4]{BM} for some progress. Speaking to its difficulty we note that even obtaining an asymptotic for $|\mathcal E(G;X)|$ is an open problem for most $G$. We refer to \cite[Introduction]{KoPa} for a discussion of results on counting number fields with various Galois groups. 
Most often one bounds by $X$ the discriminant of $K$ rather than its radical $D_K$ as we do here, see however \cite{Wood10} for some advantages of working with $D_K$.  
The conjectures in \cite{Mal} and the heuristic arguments of \cite{EV05} suggest that $|\mathcal E(G;X)| \sim \epsilon_G X \log^{\beta_G} X$ where $\epsilon_G$ is a positive real number and $\beta_G$ is a positive integer.

It may also be interesting to consider summing over $K \in \mathcal E^C(G;X)$ some other functions of the number of ramified primes in place of the indicator function of this number being equal to $d_{\lhd}(G)$, as in \cref{BMconjecture}.
For example, one may consider the M\"obius function of $D_K$ namely \[\mu(D_K) = (-1)^{|\operatorname{ram}(K)|}.\] 
\begin{conj} \label{MobiusConjecture}
As $X \to \infty$ we have
\begin{equation*} 
\sum_{K \in \mathcal E^C(G;X)} (-1)^{|\operatorname{ram}(K)|} = o(|\mathcal E^C(G;X)|).
\end{equation*}
\end{conj}
The conjecture predicts that the number of ramified primes, for number fields $K \in \mathcal E^C(G;X)$, is even approximately as often as it is odd.

The reason for us to emphasize specifically the M\"obius function is its close relation with the aforementioned indicator function from \cref{BMconjecture}.
\cite{BSS} shows that, for every integer $m \geq 2$, there exists an $S_m$-extension $K_m/\mathbb Q$ such that the number of primes of $\mathbb Q$ ramified in $K_m$ is at most $4$, falling just a little short of the conjecture from \cite{BM} predicting the existence of an $S_m$-extension ramified at only $d_{\lhd}(S_m) = 1$ prime.  
The problem of finding such an extension, and the use of sieve theory in \cite{BSS}, is subject to the parity barrier making it challenging to even produce an $S_m$-extension ramified at an odd number of primes, let alone a single prime. A sieve is also employed in \cite{TT} to obtain a lower bound on the number of $S_3$-extensions of $\mathbb Q$ with no more than $3$ ramified primes, and $S_4$-extensions with no more than $8$ ramified primes. Here, parity is one of the barriers to obtaining a lower bound of the right order of magnitude.
Another application of sieve theory in this context is \cite[Theorem 7.11]{BG} obtaining an upper bound in \cref{BMconjecture} for $G = S_4$.
Parity is one of the barriers to improving this upper bound.
Because of this, and because of the ability to express the indicator function of the prime numbers using the M\"obius function, we view \cref{MobiusConjecture} as a step toward \cref{BMconjecture}.

We would like to state analogs over $\F_q(T)$ of \cref{BMconjecture} and \cref{MobiusConjecture}. For that, we recall that the norm of a nonzero polynomial $D \in \F_q[T]$ is given by
$
|D| = |\F_q[T]/(D)| = q^{\deg D}.
$
For a finite extension $K/\F_q(T)$ we put 
\[
\operatorname{ram}(K) = \{ P \in \F_q[T] : P \ \text{is a monic irreducible polynomial ramified in} \ K\}, \quad D_K = \prod_{P \in \operatorname{ram}(K)} P.
\]


Our analogs over $\F_q(T)$ will be modeled on localized versions of \cref{BMconjecture} and \cref{MobiusConjecture}, as discussed in \cref{LocalizedRmk}.
To state these analogs, we need more notation.

\begin{defi}

Let $p$ be a prime number not dividing $|G|$, let $q$ be a power of $p$, and let $n_1, \dots, n_{d_\lhd(G)}$ be positive integers.
Let $\mathcal E_q^C\left(G;n_1, \dots, n_{d_\lhd(G)}\right)$ be the family of regular $G$-extensions $K$ of $\F_q(T)$ split completely at $\infty$ and satisfying the following two conditions.
\begin{itemize}

\item Every generator of each nontrivial inertia subgroup of $\operatorname{Gal}(K/\F_q(T))$ lies in $C$.

\item For $1 \leq j \leq d_\lhd(G)$ let $D_K(j)$ be the product of all the $P \in \operatorname{ram}(K)$ for which the generators of the inertia subgroups of $\operatorname{Gal}(K/\F_q(T))$ for primes of $K$ lying over $P$ belong to $C_j$. Then \[\deg D_K(j) = n_j.\]

\end{itemize}

\end{defi}

For every $K \in \mathcal E_q^C\left(G;n_1, \dots, n_{d_\lhd(G)}\right)$ we have
\[
D_K =  \prod_{j=1}^{d_\lhd(G)} D_K(j).
\]

As in the number field case, the family $\mathcal E_q^C\left(G;n_1, \dots, n_{d_\lhd(G)}\right)$ is finite.
\cite{ETW17} provides upper bounds of the (conjecturally) right order of magnitude on $\left|\mathcal E_q^C\left(G;n_1, \dots, n_{d_\lhd(G)}\right)\right|$ for all $q$ larger than a certain quantity depending on $C$. 
In case $C \cap H$ is either empty or a single conjugacy class of $H$ for every subgroup $H$ of $G$, near-optimal upper and lower bounds on $|\mathcal E_q^C(G;n)|$ are obtained in \cite{EVW}. 

We also recall that the zeta function of $\F_q[T]$ is given for $s \in \mathbb C$ by 
\[
\zeta_q(s) = \sum_{\substack{f \in \F_q[T] \\ f \text{ is monic}}} |f|^{-s} = \frac{1}{1 - q^{1-s}}, \quad \textup{Re}(s) > 1.
\]

\begin{conj} \label{BMconjectureFF}

Fix a prime power $q$ coprime to $|G|$.
Then there exists a positive real number $\delta^{G,C}_q$ such that as $n_1, \dots, n_{d_\lhd(G)} \to \infty$ we have
\[
\sum_{K \in \mathcal E_q^C\left(G;n_1, \dots, n_{d_\lhd(G)}\right)} \mathbf{1}_{|\operatorname{ram}(K)| = d_{\lhd}(G)} \sim \delta^{G,C}_q \cdot \prod_{j=1}^{d_\lhd(G)} \frac{\zeta_q(2)}{n_j} \cdot \left|\mathcal E_q^C\left(G;n_1, \dots, n_{d_\lhd(G)}\right)\right|.
\]
Moreover, as soon as at least one of the $n_j$ tends to $\infty$ we have
\[
\sum_{K \in \mathcal E_q^C\left(G;n_1, \dots, n_{d_\lhd(G)}\right)} (-1)^{|\operatorname{ram}(K)|} =  o\left( \left| \mathcal E_q^C\left(G;n_1, \dots, n_{d_\lhd(G)}\right) \right|\right).
\]

\end{conj}

The factor $\frac{\zeta_q(2)}{n_j} = \frac{q}{(q-1)n_j}$ is (a good approximation for) the probability that a uniformly random monic squarefree polynomial of degree $n_j \geq 2$ over $\F_q$ is irreducible.

Using the (very) special cases of Schinzel's Hypothesis H and the Chowla conjecture established in \cite{SS20}, it is perhaps possible to make partial progress on \cref{BMconjectureFF} for certain groups. 
It is however not clear to us how to make additional (or significant) progress on \cref{BMconjectureFF}, going beyond what's known about \cref{BMconjecture}. 


%
%


In this paper we prove a large finite field version of \cref{BMconjectureFF}. 
In the large finite field regime, instead of fixing $q$ and taking the $n_j$ to $\infty$, we fix the $n_j$ and take $q$ to $\infty$.
One indication that this regime is more tractable is given by \cite[Proof of Theorem 1.4]{LWZB} that obtains (among other things) the leading term of the asymptotic for $\left|\mathcal E_q^C\left(G;n_1, \dots, n_{d_\lhd(G)}\right)\right|$.
Another indication of the tractability of the large finite field regime is given by the resolution (in this setting) of very difficult problems in analytic number theory, such as Schinzel's Hypothesis H, see \cite{Ent} and references therein.


We have $\lim_{q \to \infty} \zeta_q(2) = 1$, and we believe that the arithmetic correction constants $\delta_q^{G,C}$ also converge to $1$ as $q \to \infty$.
This belief is based in part on the convergence to $1$ of various singular series constants over $\F_q[T]$ as $q \to \infty$, see also \cite{Ent}.
Our main result provides further evidence for this belief.

\begin{thm} \label{MobiusVonMangoldtConjectureLargeFF}

Fix $n_1, \dots, n_{d_\lhd(G)}$ such that $n_j > |C_j|$ for some $1 \leq j \leq d_\lhd(G)$.
Then as $q \to \infty$ along prime powers coprime to $|G|$ we have
\[
\sum_{K \in \mathcal E_q^C\left(G;n_1, \dots, n_{d_\lhd(G)}\right)} (-1)^{|\operatorname{ram}(K)|} =  o\left(\left|\mathcal E_q^C\left(G;n_1, \dots, n_{d_\lhd(G)}\right)\right|\right).
\]

Fix $n_1, \dots, n_{d_\lhd(G)}$ large enough. 
Then as $q \to \infty$ along prime powers coprime to $|G|$ we have
\[
\sum_{K \in \mathcal E_q^C\left(G;n_1, \dots, n_{d_\lhd(G)}\right)} \mathbf{1}_{|\operatorname{ram}(K)| = d_{\lhd}(G)} \sim 
\frac{\left|\mathcal E_q^C\left(G;n_1, \dots, n_{d_\lhd(G)}\right)\right|}{n_1 \cdots n_{d_\lhd(G)}}.
\]

\end{thm}

By `large enough' we mean larger than a certain function of $|G|$ (or in fact, a function of $|C|$).
This dependence on $|G|$ obtained from our argument (or rather from arguments of Conway, Parker, Fried, and V\"olklein) is ineffective, but it is likely possible to give a less elementary form of the argument that is effective. We do not pursue this direction in the current work.

The error term with which we obtain the asymptotics in \cref{MobiusVonMangoldtConjectureLargeFF} is 
\[
O\left(\frac{\left|\mathcal E_q^C\left(G;n_1, \dots, n_{d_\lhd(G)}\right)\right|}{\sqrt q}\right)
\]
where the implied constant depends $n_1, \dots, n_{d_\lhd(G)}$.
With the geometric setup in the proof of \cref{MobiusVonMangoldtConjectureLargeFF} we are well-positioned to explicate (and improve) this dependence by a study of Betti numbers, but we do not do it in this paper.
We also do not attempt to improve the dependence of the error term on $q$ by studying cohomology beyond the top degree.

Our proof of \cref{MobiusVonMangoldtConjectureLargeFF} allows us to obtain the frequency with which $D_K$ attains any given factorization type (such as the product of $d_\lhd(G)$ irreducible polynomials). 
Put differently, we can obtain the asymptotic for the sum over $K \in \mathcal E_q^C\left(G;n_1, \dots, n_{d_\lhd(G)}\right)$ of any factorization function of $D_K$ (such as $\mu(D_K)$).
For a more detailed discussion of factorization types and factorization functions see \cite{Gor20}.
We can then see that the factors $D_K(j)$ of $D_K$ indeed behave as random independent monic squarefree polynomials of degree $n_j$, as far as their factorizations into monic irreducible polynomials are concerned.

\section{Sketch of a Proof of \cref{MobiusVonMangoldtConjectureLargeFF}}

We describe a simplified variant of the argument we use to prove \cref{MobiusVonMangoldtConjectureLargeFF}.
In this sketch, for simplicity we will mostly restrict to the special case $d_\lhd(G) = 1$.
Towards the very end of our sketch, we will comment on what the actual argument is, and mention some of the additional difficulties involved.

For a squarefree monic polynomial $f$ of degree $n \geq 1$ over $\F_q$ we denote by $\sigma_f$ the conjugacy class in $S_n$ of the permutation raising the roots of $f$ to $q$th power. The lengths of the cycles of $\sigma_f$ are the degrees of the irreducible factors of $f$. 
Let us restate this in a more geometric language.

Let $\textup{Conf}^{n}$ be the configuration space of $n$ unordered distinct points on the affine line over $\overline{\F_q}$.
Viewing these $n$ points as the roots of a monic squarefree polynomial of degree $n$ over $\overline{\F_q}$, one endows $\textup{Conf}^{n}$ with the structure of a variety over $\F_q$, namely
\[
\textup{Conf}^{n} = \{(a_0, \dots, a_{n-1}) : \operatorname{Discriminant}_T(a_0 + a_1 T + \dots + a_{n-1}T^{n-1} + T^n) \neq 0 \}.
\]
There exists a continuous homomorphism $\lambda$ from the (profinite) \'etale fundamental group $\pi_1^{\textup{\'et}}(\textup{Conf}^{n})$ to $S_n$ such that for every $f \in \textup{Conf}^{n}(\F_q)$ (viewed as a monic squarefree polynomial of degree $n$ over $\F_q$) the conjugacy class in $S_n$ of the value of $\lambda$ at the element $\operatorname{Frob}_{f} \in \pi_1^{\textup{\'et}}(\textup{Conf}^{n})$ is $\sigma_f$.

In fact, this definition of $\textup{Conf}^n$ works over an arbitrary field, and slightly abusing notation we denote by $\textup{Conf}^n(\mathbb C)$ the configurations space of $n$ complex numbers (equivalently, monic polynomials of degree $n$ over $\mathbb C$ with no repeated roots). This space is a manifold, and its fundamental group is the braid group on $n$ strands, namely
\[
\pi_1(\textup{Conf}^n(\mathbb C)) = B_n = \langle \sigma_1, \dots, \sigma_{n-1} : \sigma_i \sigma_j = \sigma_j \sigma_i \text{ for } i > j+1, \ \sigma_i \sigma_{i+1} \sigma_i = \sigma_{i+1} \sigma_i \sigma_{i+1} \text{ for } i < n-1 \rangle.
\]
The counterpart of the homomorphism $\lambda \colon \pi_1^{\textup{\'et}}(\textup{Conf}^{n}) \to S_n$ in this setting is the homomorphism from $B_n$ to $S_n$ that maps the generator $\sigma_i$ to the transposition $(i \ i+1) \in S_n$ for every $1 \leq i \leq n-1$.

For every prime power $q$ coprime to $|G|$ there exists an $\F_q$-variety $\mathsf{Hur}_{G,C}^n$ parametrizing $G$-covers of $\mathbb P^1$ branched at $n$ points with monodromy of type $C$ and a choice of a point over $\infty$.
To simplify matters in this sketch, we assume that $\mathsf{Hur}_{G,C}^n$ is geometrically connected (this assumption is not satisfied for many choices of $G$, $C$, and $n$, and when it fails to hold, the connected components of $\mathsf{Hur}_{G,C}^n$ are not necessarily defined over $\F_q$).
We make the identification
\[
\mathsf{Hur}_{G,C}^n(\F_q) = \mathcal E_q^C(G;n).
\]

There is a finite \'etale map $\mathsf{Hur}_{G,C}^n \to \textup{Conf}^{n}$ that on the level of $\F_q$-points sends every $K \in \mathcal E_q^C(G;n)$ to $D_K$.
Therefore the distribution of the factorization type of $D_K$ as $K$ ranges over $\mathcal E_q^C(G;n)$, or rather the distribution of $\sigma_{D_K}$ among the conjugacy classes of $S_n$, is the distribution of $\lambda(\operatorname{Frob}_{D_K})$ among the conjugacy classes of $S_n$ as $K$ ranges over $\mathsf{Hur}_{G,C}^n(\F_q)$.  

In view of our simplifying assumption, the variety $\mathsf{Hur}_{G,C}^n$ is a connected finite \'etale cover of $\textup{Conf}^{n}$, so we can view $\pi_1^{\textup{\'et}}(\mathsf{Hur}_{G,C}^n)$ as an open subgroup of $\pi_1^{\textup{\'et}}(\textup{Conf}^{n})$. By a version of Chebotarev's density theorem, the aforementioned distributions are governed by the image $H$ of $\pi_1^{\textup{\'et}}(\mathsf{Hur}_{G,C}^n)$ in $S_n$ under the homomorphism $\lambda \colon \pi_1^{\textup{\'et}}(\textup{Conf}^{n}) \to S_n$.
That is, for every conjugacy class $\Delta$ of $S_n$, as $q \to \infty$ along prime powers coprime to $|G|$ we have
\[
|\{K \in \mathcal E_q^C(G;n) : \sigma_{D_K} = \Delta \}|  = |\{K \in \mathsf{Hur}_{G,C}^n(\F_q)  : \lambda(\operatorname{Frob}_{D_K}) \in \Delta \}|  \sim \frac{|\Delta \cap H|}{|H|} \cdot |\mathcal E_q^C(G;n)|.
\]
For our purposes, it would be sufficient to show that $H = S_n$ for $n$ large enough.
The subgroup $H$ does not change if we work with \'etale fundamental groups over $\overline{\F_q}$ rather than over $\F_q$ as we did so far.

The group $B_n$ acts on $C^n$ (the set of $n$-tuples of elements from $C$) from the right by 
\begin{equation} \label{BraidsInAction}
(c_1, \dots, c_{i-1}, c_i, c_{i+1}, c_{i+2} \dots, c_n)^{\sigma_i} = (c_1, \dots, c_{i-1}, c_{i+1}, c_i^{c_{i+1}}, c_{i+2} \dots, c_n), \quad 1 \leq i \leq n-1,
\end{equation}
where $c_1, \dots, c_n \in C$ and for group elements $g,h$ we use the notation $g^h$ for $h^{-1} g h$.
For (an appropriately chosen) $s \in C^n$ whose entries generate $G$, the stabilizer of $s$ in $B_n$ is the counterpart of the subgroup $\pi_1^{\textup{\'et}}(\mathsf{Hur}_{G,C}^n)$ of $\pi_1^{\textup{\'et}}(\textup{Conf}^{n})$.
In particular, one can show that $H$ is the image in $S_n$ of the stabilizer of $s$ in $B_n$.
At this point we have reduced the arithmetic problem we wanted to solve to a group-theoretic question.
This kind of reduction is by now standard, being emplyed frequently in the proofs of results in the large finite field limit, see for instance \cite{Ent}, and \cite{LWZB}. 

The technical heart of this paper lies in showing that for $n$ large enough we indeed have $H = S_n$.
Recalling that the pure braid group $PB_n$ is the kernel of the homomorphism from $B_n$ to $S_n$, we see that the equality $H = S_n$
is equivalent to the transitivity of the action of $PB_n$ on the orbit of $s$ under the action of $B_n$.
As a result, our work gives an additional justification for \cite[Heuristic 4.7]{BE} suggesting the transitivity of a very similar action of a (profinite) group closely related to $PB_n$.

To state our technical result in greater generality, we need to recall a few additional notions.
\begin{defi}

A rack is a set $X$ with a binary operation $x^y$ for $x,y \in X$ such that for every $y \in X$ the function $x \mapsto x^y$ is a bijection from $X$ to $X$,
and for every $x,y,z \in X$ we have $(z^x)^y = (z^{y})^{x^y}$.
If in addition $x^x = x$ for all $x \in X$ then we say that $X$ is a quandle.

\end{defi}

As an example of a quandle we can take $C$ with the binary operation of conjugation. For a rack $X$ we have a right action of $B_n$ on $X^n$ using the formula in \cref{BraidsInAction}.

\begin{defi} \label{DefRackComponents}

Let $X$ be a finite rack. 
Define the (directed unlabeled) Schreier graph of $X$ to be the graph whose set of vertices is $X$ and whose set of edges is $\{(x, x^y) : (x,y) \in X \times X\}$.
We say that $X$ is connected if its Schreier graph is connected, and we call the connected components of the Schreier graph of $X$ simply `the connected components of $X$'.

\end{defi}

The Schreier graph may contain loops. Two vertices in the Schreier graph are weakly connected if and only if they are strongly connected.
The quandle $C$ is an example of a connected rack, and the Schreier graph of $C$ is the Schreier graph of the right action of $G$ on $C$ by conjugation, where $C$ also plays the role of a generating set of $G$.

\begin{defi} \label{DefSubRackGen}

Let $X$ be a finite rack. We say that a subset $X_0$ of $X$ is a subrack if for all $x,y \in X_0$ we have $x^y \in X_0$. 
We say that elements $x_1, \dots, x_n \in X$ generate $X$ if there is no proper subrack of $X$ containing $x_1, \dots, x_n$.

\end{defi}

Our technical result is `an $H = S_n$ theorem', but in the generality of racks.

\begin{thm} \label{ConnectedLargeMonodromy}

Let $X$ be a connected finite rack. 
Let $n$ be a sufficiently large positive integer, and let $(x_1, \dots, x_n) \in X^n$ be an $n$-tuple of elements from $X$ such that $x_1, \dots, x_n$ generate $X$.
Denote by $\operatorname{Stab}(x_1, \dots, x_n)$ the stabilizer in $B_n$ of $(x_1, \dots, x_n)$.
Then the image $H$ of $\operatorname{Stab}(x_1, \dots, x_n)$ under the homomorphism from $B_n$ to $S_n$ is either $A_n$ or $S_n$.
If moreover $X$ is a quandle, then $H = S_n$.

\end{thm}

\cref{TwoElementExample} shows that it is necessary to distinguish between racks and quandles for that matter.
For the possibility of extending \cref{ConnectedLargeMonodromy} to biracks and to more general algebraic structures, see \cref{MoreGeneralThanRacks}.

We shall now briefly describe a proof strategy for \cref{ConnectedLargeMonodromy}.
At first we prove in \cref{HomogeneousGroupsAreBig} a criterion for a subgroup of $S_n$ to be either $A_n$ or $S_n$ - the subgroup has to be homogeneous (in the sense of \cref{HomDef}) of every possible degree. The proof of this criterion rests on a theorem of Jordan about permutation groups, and somewhat unexpectedly, on (a slightly strengthened form of) Bertrand's postulate - the existence of a prime number between $n/2$ and $n$. 
 
Let $\mathcal T_2$ be the trivial rack on two elements, and consider the rack $Z = X \times \mathcal T_2$.
As we show in the proof of \cref{OptConj}, in order to show that the subgroup $H$ from \cref{ConnectedLargeMonodromy} has the required homogeneity property,
it suffices to show that certain elements in $Z^n$ lie in the same orbit under the action of $B_n$ in case their projections to $X^n$, and their projections to $\mathcal T_2^n$, lie in the same orbit. 
In \cref{ActionOnPairsBraidGroup} we show that this is indeed the case, even for a more general rack $Y$ in place of $\mathcal T_2$, 
assuming that $Y$ satisfies certain natural (yet somewhat technical) conditions.

The problem of understanding the orbits of the action of $B_n$ on $Z^n$ in case $Z$ is a union of conjugacy classes in a group has been addressed by several authors including Conway, Parker, Fried, and V\"olklein, see \cite[Theorem 3.1]{Wood21}.
In \cref{CPFW} we recast their arguments in the generality of racks, and this leads us to a proof of \cref{ActionOnPairsBraidGroup}.
What remains is to check that the conditions of \cref{ActionOnPairsBraidGroup} are satisfied in case $Y = \mathcal T_2$. 
This is done in \cref{TrivialProductivity} that deals with commutators in the structure group of a rack, see \cref{DefRackGroup}.

The proof of \cref{ConnectedLargeMonodromy} just sketched, in particular \cref{ActionOnPairsBraidGroup}, is potentially applicable to statistics in the large finite field limit of arithmetic functions arising from other racks, not necessarily directly related to factorization types or to counting covers of $\mathbb P^1$ with certain properties. 
More precisely, \cref{ActionOnPairsBraidGroup} could be helpful for studying correlation sums of arithmetic functions associated to $X$ and to $Y$.
We do not elaborate here on the possibility of associating an arithmetic function to a rack, or to more general algebraic structures.

A downside of this proof of \cref{ConnectedLargeMonodromy}, is that even if the arguments of Conway--Parker--Fried--V\"olklein will be made effective, the resulting dependence of $n$ on $|X|$ will be suboptimal. We have therefore included in this paper also a more direct proof of the quandle case of \cref{ConnectedLargeMonodromy} which is likely to give a better dependence of $n$ on $|X|$. In this proof, a different criterion for a permutation group to be $S_n$ is used, based on invariable generation - see \cref{InvariableGnerationProp}.

We also include two effective forms of \cref{ConnectedLargeMonodromy} in two special cases.
In the first special case $X$ is a (certain conjugacy class in a) finite simple group $G$.
In this case we are only able to show that $H$ contains an $n$-cycle, and even that under an additional simplification - we consider the action of $B_n$ on $G \backslash X^n$ rather than on $X^n$, see \cref{FiniteSimpleGroup}.
This additional simplification means that as an application, we can count $G$-extensions of $\F_q(T)$ unramified at $\infty$ rather than split completely at $\infty$.
Since we only know that $H$ contains an $n$-cycle (and not that $H = S_n$), we do not know exactly how often $D_K$ is irreducible, but we should be able to show that this happens with positive probability as $q \to \infty$ along prime powers coprime to $|G|$. 
For the second special case see \cref{SmSmalln}. 
These results are in the spirit of \cite{Chen} counting connected components of generalized Hurwitz spaces where we do not fix $G$ and let $n$ grow, but rather vary $G$ in a certain family of finite groups while $n$ remains reasonably small.

Our proof of \cref{MobiusVonMangoldtConjectureLargeFF} is given in \cref{THEproof} and it does not quite proceed by invoking Chebotarev's theorem as in this sketch,
rather it adapts a proof of Chebotarev's theorem to the special case at hand by expressing the indicator functions of conjugacy classes as linear combinations of complex characters.   
Put differently, we interpret certain factorization functions (such as $\mu(D_K)$) as trace functions of sheaves on $\mathsf{Hur}_{G,C}^n$, and apply the Grothendieck--Lefschetz fixed-point formula in conjunction with Deligne's Riemann Hypothesis to estimate the sums of these trace functions.
The advantage of repeating a proof of Chebotarev's theorem (over a use of the theorem as a black box) is that further progress on some of the function field problems mentioned in the introduction is thus reduced to questions on the cohomology of local systems on $\textup{Conf}^n$.

In order to remove the assumption $d_\lhd(G) = 1$ and deal with an arbitrary nontrivial finite group $G$, we prove in \cref{OptConj} and \cref{SqfreeCorSn} a generalization of \cref{ConnectedLargeMonodromy} to finite racks with $k \geq 1$ connected components. For this generalization, instead of showing that a certain subgroup of a symmetric group is large, we need to consider subgroups of a direct product of symmetric (and alternating) groups. Building on \cite{BSGKS} and on Goursat's Lemma, we provide in \cref{SubsDirectProdGroups} a criterion for a subgroup of a direct product of groups to be the whole group, valid under certain assumptions on the groups in the product.

\section{Groups and Their Actions} \label{GroupActions}


\subsection{Braid Group and its Action on Presentations}


Let $F_n$ be the free group on the letters $x_1, \dots, x_n$.
We can view $B_n$ as the subgroup of $\mathrm{Aut}(F_n)$ generated by the automorphisms
\[
\sigma_i(x_j) = 
\begin{cases}
x_j &j \notin \{i,i+1\} \\
x_{i+1} &j = i \\
x_i^{x_{i+1}} &j = i+1,
\end{cases}
\quad \quad 1 \leq i \leq n-1.
\]

Let $C \subseteq G$ be a generating set of $G$ which is a disjoint union of some conjugacy classes $C_1, \dots, C_k$ of $G$.
Let $\operatorname{Mor}^C(F_n,G)$ be the collection of all homomorphisms $\theta \colon F_n \to G$ with $\theta(x_1), \dots, \theta(x_n) \in C$,
and put $\operatorname{Mor}(F_n,G) = \operatorname{Mor}^G(F_n,G)$.
The group $\mathrm{Aut}(F_n)$ acts on $\operatorname{Mor}(F_n,G)$ from the right by precomposition, namely 
\[
\theta^\phi = \theta \circ \phi, \quad \theta \in \operatorname{Mor}(F_n,G), \ \phi \in \mathrm{Aut}(F_n).
\]
Restricting this action to $B_n$, we also get a right action of $B_n$ on the subset $\operatorname{Mor}^C(F_n,G)$ of $\operatorname{Mor}(F_n,G)$.

Let $n_1, \dots, n_k$ be nonnegative integers such that $n_1 + \dots + n_k = n$.
We further get an action of $B_n$ on the subsets
\[
\operatorname{Sur}^C(F_n,G) = \{\theta \in \operatorname{Mor}^C(F_n,G) : \theta \ \text{is surjective}\}, \
\operatorname{Sur}^C_1(F_n,G) = \{ \theta \in \operatorname{Sur}^C(F_n,G) : \theta(x_1) \cdots \theta(x_n) = 1 \},
\]
\[
\operatorname{Sur}^C(F_n,G;n_1, \dots, n_k) = \{ \theta \in \operatorname{Sur}^C(F_n,G) : |\{1 \leq i \leq n : \theta(x_i) \in C_j\}| = n_j \text{ for all } 1 \leq j \leq k \},
\]
\[
\operatorname{Sur}^C_1(F_n,G;n_1, \dots, n_k) = \{ \theta \in \operatorname{Sur}^C_1(F_n,G) : |\{1 \leq i \leq n : \theta(x_i) \in C_j\}| = n_j \text{ for all } 1 \leq j \leq k \}.
\]
We will simply write $\operatorname{Sur}(F_n,G)$ for $\operatorname{Sur}^G(F_n,G)$ and $\operatorname{Sur}_1(F_n,G)$ for $\operatorname{Sur}^G_1(F_n,G)$.

We identify $\operatorname{Mor}(F_n,G)$ with $G^n$ by sending $\theta \in \operatorname{Mor}(F_n,G)$ to $(\theta(x_1), \dots, \theta(x_n)) \in G^n$, so that $\operatorname{Mor}^C(F_n,G)$ is identified with $C^n$.
Every $\phi \in \mathrm{Aut}(F_n)$ gives us the $n$ words
$
w_i(x_1, \dots, x_n) = \phi(x_i) \in F_n
$
in the letters $x_1, \dots, x_n$. Given an $n$-tuple $(g_1, \dots, g_n) \in G^n$, the action of $\phi$ on it is
\[
(g_1, \dots, g_n)^\phi = (w_1(g_1, \dots, g_n), \dots, w_n(g_1, \dots, g_n)).
\] 
In case $\phi = \sigma_i$ for some $1 \leq i \leq n-1$ is one of our generators of $B_n$, we recover \cref{BraidsInAction}.

With this identification we also have $\operatorname{Sur}^C(F_n, G) = \{(g_1, \dots, g_n) \in C^n : \langle g_1, \dots, g_n \rangle = G \}$ and
\[
\operatorname{Sur}^C_1(F_n, G) = \{(g_1, \dots, g_n) \in C^n : \langle g_1, \dots, g_n \rangle = G, \ g_1 \cdots g_n = 1 \}.
\]
Moreover, we identify $\operatorname{Sur}^C_1(F_n, G; n_1, \dots, n_k)$ with
\[
\{(g_1, \dots, g_n) \in C^n : \langle g_1, \dots, g_n \rangle = G, \ g_1 \cdots g_n = 1, \ |\{1 \leq i \leq n : g_i \in C_j\}| = n_j \text{ for all } 1 \leq j \leq k \}.
\]



For disjoint subsets $D_j \subseteq \{1, \dots, n\}$ with $|D_j| = n_j$ for $1 \leq j \leq k$, we make the identification
\[
\{\sigma \in S_n : \sigma(D_j) = D_j \text{ for every } 1 \leq j \leq k \} = S_{n_1} \times \dots \times S_{n_k}.
\]
Often we take $D_j = \{n_1+ \dots + n_{j-1}+1, \dots, n_1 + \dots + n_j\}$.
We denote by $B_{n_1, \dots, n_k}$ the inverse image of $S_{n_1} \times \dots \times S_{n_k}$ under the homomorphism from $B_n$ to $S_n$.
Sometimes $B_{n_1, \dots, n_k}$ is called a colored braid group (with $k$ colors).

For every $1 \leq m < n$, we identify the subgroup of $B_n$ generated by $\sigma_1, \dots, \sigma_{m-1}, \sigma_{m+1}, \dots, \sigma_{n-1}$ with $B_m \times B_{n-m}$.

\subsection{Permutation Groups}

Let $\Gamma$ be a group acting on a set $S$ from the right. We say that $B \subseteq S$ is a block of $\Gamma$ if for every $g \in \Gamma$ either $B^g = B$ or $B^g \cap B = \emptyset$. We call $H = \{g \in \Gamma : B^g = B\}$ the stabilizer of the block, and note that it acts on $B$.

\begin{prop} \label{BijectionFromBlocks}

The natural map $B/H \to S/\Gamma$ is injective,
so in case $B$ meets every orbit of $\Gamma$ this map is bijective.

\end{prop}

\begin{proof}

Let $b_1,b_2 \in B$ be representatives for orbits under the action of $H$ whose images in $S/\Gamma$ coincide. 
This means that there exists $g \in \Gamma$ with $b_1^g = b_2$. In particular, $B^g \cap B \neq \emptyset$ so $g \in H$ since $B$ is a block of $\Gamma$.
It follows that $b_1$ and $b_2$ are in the same orbit under the action of $H$ so the required injectivity is established.
\end{proof}

\begin{defi} \label{HomDef}

Let $k \leq n$ be nonnegative integers.
We say that a subgroup $H$ of the symmetric group $S_n$ is $k$-homogeneous if the action of $H$ on the set of all $k$-element subsets of $\{1, \dots, n\}$ is transitive.  

\end{defi}

\begin{prop} \label{HomogeneousGroupsAreBig}

Let $H$ be an $\lfloor n/2 \rfloor$-homogenous subgroup of $S_n$. Then either $H = A_n$ or $H = S_n$.

\end{prop}

\begin{proof}

For small $n$ this can be verified by hand. For $n$ large enough, we can find a prime number $p$ satisfying $(n+1)/2 < p \leq n-3$.
The stabilizer of $\{1, \dots, \lfloor n/2 \rfloor\}$ in the action of $S_n$ on the set of all $\lfloor n/2 \rfloor$-element subsets of $\{1, \dots, n\}$ is the subgroup $S_{\lfloor n/2 \rfloor} \times S_{\lceil n/2 \rceil}$.
The $\lfloor n/2 \rfloor$-homogeneity of $H$ is tantamount to the equality $H \cdot \left(S_{\lfloor n/2 \rfloor} \times S_{\lceil n/2 \rceil}\right) = S_n$ of subsets of $S_n$. Therefore
\[
n! = |S_n| = \left|H \cdot \left(S_{\lfloor n/2 \rfloor} \times S_{\lceil n/2 \rceil}\right)\right| = \frac{|H| \cdot |S_{\lfloor n/2 \rfloor} \times S_{\lceil n/2 \rceil}|}{\left|H \cap \left(S_{\lfloor n/2 \rfloor} \times S_{\lceil n/2 \rceil}\right)\right|} = \frac{|H| \cdot \lfloor n/2 \rfloor! \cdot \lceil n/2 \rceil!}{\left|H \cap \left(S_{\lfloor n/2 \rfloor} \times S_{\lceil n/2 \rceil}\right)\right|}
\]
so $n!$ divides $|H| \cdot \lfloor n/2 \rfloor! \cdot \lceil n/2 \rceil!$ hence $p$ divides this number as well.

Our choice of $p$ guarantees that $p$ divides $|H|$. By Cauchy's theorem, $H$ contains an element of order $p$.
An element of order $p$ in $S_n$ is a product of $p$-cycles, but $p> n/2$ so in our case this element is necessarily a $p$-cycle.
It is readily checked that an $\lfloor n/2 \rfloor$-homogeneous subgroup of $S_n$ is transitive (equivalently, $1$-homogeneous) and moreover primitive (or even $2$-homogeneous). By a theorem of Jordan, the only primitive subgroups of $S_n$ that contain a $p$-cycle for a prime number $p \leq n-3$ are $A_n$ and $S_n$.
We conclude that either $H = A_n$ or $H = S_n$ as required.
\end{proof}

\begin{defi}

Let $I$ be an indexing set, and let $\{H_i\}_{i \in I}$ be subgroups of a group $H$. We say that these subgroups invariably generate $H$ if for every choice of elements $\{\sigma_i\}_{i \in I}$ from $H$, the conjugate subgroups $\{H_i^{\sigma_i}\}_{i \in I}$ generate the group $H$. 

\end{defi}

For any fixed $i \in I$ in the above definition we can assume, without loss of generality, that $\sigma_i = 1$.

\begin{prop} \label{InvariableGnerationProp}

Let $n, k$ be positive integers with $n > 2k$, and let $\sigma \in S_n$ be a permutation all of whose cycles are of lengths exceeding $k$. Then the subgroups $S_{n-k}$ and $\langle \sigma \rangle$ invariably generate $S_n$.

\end{prop}

\begin{proof}

We view $S_{n-k}$ as the group of permutations of $\{1, \dots, n-k\}$ in $S_{n}$ fixing each element of $\{n-k+1, \dots, n\}$.
Denote by $H$ the subgroup of $S_n$ generated by $S_{n-k}$ and a conjugate $\rho$ of $\sigma$.
We need to show that $H = S_n$.
First we claim that $H$ acts transitively on $\{1, \dots, n\}$.
We take $j \in \{1, \dots, n\}$ and our task is to show that it lies in the orbit of $1$ under the action of $H$.
If $j \in \{1, \dots, n-k\}$ this is clear since $H$ contains $S_{n-k}$, so we assume that $j \in \{n-k+1, \dots, n\}$.
Since the cycle of $\rho$ in which $j$ lies is of length more than $k$, it contains an element from $\{1, \dots, n-k\}$.
As all the powers of $\rho$ lie in $H$, we conclude that $j$ is indeed in the orbit of $1$ under the action of $H$, as required for transitivity.

Next, we claim that $H$ is primitive. 
Toward a contradiction, suppose that $\{1, \dots, n\}$ can be partitioned into disjoint blocks for $H$ with at least two distinct blocks, each block of size at least $2$.
As $n-k > n/2$ by assumption, and the size of every block is a proper divisor of $n$, we conclude that $\{1, \dots, n-k\}$ is not contained in a single block, 
and some block $B$ meets $\{1, \dots, n-k\}$ at two distinct elements at least, say $i$ and $j$.
Since $\{1, \dots, n-k\} \nsubseteq B$, we can find $\tau \in S_{n-k}$ with $\tau(i) = i \in B$ and $\tau(j) \notin B$.
We see that $\tau \in H$, that $\tau B \cap B \neq \emptyset$, and that $\tau(B) \neq B$. This contradicts our assumption that $B$ is a block for $H$, and concludes the proof of primitivity.

As $n - k > k \geq 1$, there exists a transposition in $S_{n-k}$, so $H$ contains a transposition.
By a theorem of Jordan, the only primitive subgroup of $S_n$ that contains a transposition is $S_n$ itself so $H = S_n$ as required.
\end{proof}

\subsection{Direct Products of Groups}

\begin{lem} \label{SubsDirectProdGroups}

Let $G_1, \dots, G_n$ be groups such that for every $1 \leq i < j \leq n$ either $G_i \cong G_j$ or $G_i$ and $G_j$ do not have nonabelian simple isomorphic quotients.
Put $G = G_1 \times \dots \times G_n$ and let $H$ be a subgroup of $G$ satisfying the following three conditions.
\begin{itemize}

\item The restriction of the natural homomorphism $G \to G^{\textup{ab}}$ to $H$ is surjective.

\item The subgroup $H$ projects onto $G_i$ for every $1 \leq i \leq n$.

\item The subgroup $H$ projects onto $G_i \times G_j$ for every $1 \leq i < j \leq n$ with $G_i \cong G_j$.

\end{itemize}
Then $H = G$.

\end{lem}

\begin{proof}

Let $\mathcal G_1, \dots, \mathcal G_r$ be all the distinct isomorphism types appearing among the groups $G_1, \dots, G_n$.
Then $G \cong \mathcal G_1^{m_1} \times \dots \times \mathcal G_r^{m_r}$ where $m_i = |\{1 \leq j \leq n : G_j \cong \mathcal G_i\}|$ for $1 \leq i \leq r$.
We claim that $H$ projects onto $\mathcal G_i^{m_i}$ for every $1 \leq i \leq r$. Indeed, in case $m_i = 1$ this follows from the second condition that $H$ satisfies by assumption.
In case $m_i \geq 2$ this follows from \cite[Lemma 6.6]{BSGKS} applied to the projection of $H$ to $\mathcal G_i^{m_i}$, using the first and the third condition that $H$ satisfies.
The claim is thus established.

We prove by induction on $1 \leq i \leq r$ that $H$ projects onto $\mathcal G_1^{m_1} \times \dots \times \mathcal G_i^{m_i}$. 
The base case $i=1$ follows from the claim established above.  
For $i>1$ we apply Goursat's Lemma to the projection $H_i$ of $H$ in $(\mathcal G_1^{m_1} \times \dots \times \mathcal G_{i-1}^{m_{i-1}}) \times \mathcal G_i^{m_i}$.
Since $H_i$ projects onto both $\mathcal G_1^{m_1} \times \dots \times \mathcal G_{i-1}^{m_{i-1}}$ by induction, and onto $\mathcal G_i^{m_i}$ by the aforementioned claim,
there exists a group $K$ and surjective homomorphisms $\varphi \colon \mathcal G_1^{m_1} \times \dots \times \mathcal G_{i-1}^{m_{i-1}}  \to K$, $\psi \colon \mathcal G_i^{m_i} \to K$ such that $H_i = (\mathcal G_1^{m_1} \times \dots \times \mathcal G_{i-1}^{m_{i-1}}) \times_K \mathcal G_i^{m_i}$. We claim that $K$ is trivial. 

Toward a contradiction, suppose that $K$ admits a simple quotient $S$. Then we have surjections $\overline \varphi \colon \mathcal G_1^{m_1} \times \dots \times \mathcal G_{i-1}^{m_{i-1}} \to S$ and $\overline \psi \colon \mathcal G_i^{m_i} \to S$. 
Since $S$ is simple, and $\overline \varphi, \overline \psi$ map normal subgroups to normal subgroups, we see that $S$ is a quotient of $\mathcal G_t$ for some $1 \leq t \leq i-1$ and a quotient of $\mathcal G_i$ because the direct factors are normal subgroups that generate the direct product. As $\mathcal G_t \ncong \mathcal G_i$, our initial assumption implies that $S$ is abelian. 
We note that $H_i$ is contained in the proper subgroup $\overline H_i = (\mathcal G_1^{m_1} \times \dots \times \mathcal G_{i-1}^{m_{i-1}}) \times_S \mathcal G_i^{m_i}$
of $\mathcal G_1^{m_1} \times \dots \times \mathcal G_i^{m_i}$. Since $S$ is abelian, this subgroup $\overline H_i$ contains the commutator subgroup of $\mathcal G_1^{m_1} \times \dots \times \mathcal G_i^{m_i}$. It follows that the restriction of the natural homomorphism $\mathcal G_1^{m_1} \times \dots \times \mathcal G_i^{m_i} \to (\mathcal G_1^{m_1} \times \dots \times \mathcal G_i^{m_i})^{\textup{ab}}$ to $\overline H_i$, and therefore also to $H_i$, is not surjective.
This contradicts the first condition that $H$ satisfies, and thus proves the claim that $K = \{1\}$. 

We conclude that $H_i = (\mathcal G_1^{m_1} \times \dots \times \mathcal G_{i-1}^{m_{i-1}}) \times \mathcal G_i^{m_i}$ so our induction is complete.
Plugging $i=r$ we get that $H = G$, as required.
\end{proof}

\begin{cor} \label{HomogeneityProducts}

For every $1 \leq i \leq r$ let $n_i$ be an integer for which
\[
n_i \geq 7, \quad {n_i \choose \lfloor n_i/2 \rfloor} > 2^r.
\]
Let $H$ be a subgroup of $S_{n_1} \times \dots \times S_{n_r}$ such that for every choice of pairs $(X_i,Y_i)$ of subsets 
\[
X_i, Y_i \subseteq \{n_1+ \dots + n_{i-1}+1, \dots, n_1 + \dots + n_i\}, \quad |X_i| = |Y_i| = \lfloor n_i/2 \rfloor, 
\]
there exists an $h \in H$ for which $h(X_i) = Y_i$ for every $1 \leq i \leq r$. Then $H$ contains $A_{n_1} \times \dots \times A_{n_r}$.

\end{cor}

\begin{remark}

The assumption that $n_1, \dots, n_r$ are large enough is possibly unnecessary, but is satisfied in our applications, so we include it because it facilitates obtaining \cref{HomogeneityProducts} as a consequence of \cref{SubsDirectProdGroups}.

\end{remark}

\begin{proof}

Put $K = H \cap (A_{n_1} \times \dots \times A_{n_r})$.
We need to show that $K = A_{n_1} \times \dots \times A_{n_r}$ and we will do this by invoking \cref{SubsDirectProdGroups} whose assumptions we shall verify now.
The first assumption is met because the groups $A_n$ for $n \geq 5$ are nonabelian pairwise nonisomorphic simple groups (and the groups $A_n$ for $n < 5$ do not have nonabelian simple quotients).
We check next that $K$ satisfies the three conditions in \cref{SubsDirectProdGroups}.

The first condition is satisfied because the abelianization of $A_{n_1} \times \dots \times A_{n_r}$ is trivial as $n_1, \dots, n_r \geq 5$ by assumption.
To check the second condition we start by fixing $1 \leq i \leq n$ and noting that in view of our assumptions on $H$ and \cref{HomogeneousGroupsAreBig}, the projection $H_i$ of $H$ to $S_{n_i}$ contains $A_{n_i}$. Since
\[
H/K = H/ \left( H \cap (A_{n_1} \times \dots \times A_{n_r}) \right) \cong H \cdot (A_{n_1} \times \dots \times A_{n_r})/ A_{n_1} \times \dots \times A_{n_r} \leq (\mathbb Z/2 \mathbb Z)^r
\]
denoting by $K_i$ the projection of $K$ to $S_{n_i}$ we see that $H_{i}/K_i$ is an elementary abelian $2$-group.
We conclude that $K_i$ contains $A_{n_i}$ so the second condition is indeed satisfied.


To check the third condition, we take $1 \leq i < j \leq r$ with $n_i = n_j$, and denote by $K_{i,j}$ the projection of $K$ to $A_{n_i} \times A_{n_j}$.
As $n_i = n_j$ the function 
\[
f \colon \{n_1+ \dots + n_{i-1}+1, \dots, n_1 + \dots + n_i\} \to \{n_1+ \dots + n_{j-1}+1, \dots, n_1 + \dots + n_j\}, 
\]
given by
\[
f(x) = x + n_{i+1} + \dots + n_j = x + n_i + \dots + n_{j-1}
\]
is a bijection.
This bijection will be silently used in what follows to identify the group $S_{n_i}$ with the group $S_{n_j}$ (and the group $A_{n_i}$ with the group $A_{n_j}$).

Suppose toward a contradiction that $K_{i,j}$ is a proper subgroup of $A_{n_i} \times A_{n_j}$.
In view of the second condition verified above, the subgroup $K_{i,j}$ projects onto both $A_{n_i}$ and $A_{n_j}$ so since these two groups are simple, 
Goursat's Lemma tells us that $K_{i,j} = \{(\sigma, \psi(\sigma)) : \sigma \in A_{n_i}\}$ for some automorphism $\psi \colon A_{n_i} \to A_{n_j}$.
Our assumption that $n_i \geq 7$ implies that $\operatorname{Aut}(A_{n_i}) = S_{n_i}$, namely there exists $\tau \in S_{n_i}$ for which $K_{i,j} = \{(\sigma, \tau \sigma \tau^{-1}) : \sigma \in A_{n_i}\}$.

We pick an $\lfloor n_i/2 \rfloor$-element subset $X$ of $ \{n_1+ \dots + n_{i-1}+1, \dots, n_1 + \dots + n_i\}$, for instance 
\[
X =  \{n_1+ \dots + n_{i-1}+1, \dots, n_1 + \dots + n_{i-1} + \lfloor n_i/2 \rfloor \} 
\]
and consider the orbit of $(X, \tau(f(X)))$ under the action of $K_{i,j}$.
On the one hand, this orbit is $\{(\sigma(X), \tau(\sigma(f(X)))) : \sigma \in A_{n_i}\}$ so its length is at most
\[
|\{\sigma(X) : \sigma \in A_{n_i}\}| \leq {n_i \choose \lfloor n_i/2 \rfloor}.
\]
On the other hand, in view of our initial assumption on $H$, the orbit of $(X, \tau(f(X)))$ under the action of the projection $H_{i,j}$ of $H$ to $S_{n_i} \times S_{n_j}$ has length ${n_i \choose \lfloor n_i/2 \rfloor}^2$.
Since $K_{i,j}$ is a normal subgroup of $H_{i,j}$ with $[H_{i,j} : K_{i,j}] \leq [H : K] \leq 2^r$, the length of the orbit of $(X, \tau(f(X)))$ under the action of $K_{i,j}$ is at least $2^{-r} {n_i \choose \lfloor n_i/2 \rfloor}^2$. We conclude that 
\[
2^{-r}{n_i \choose \lfloor n_i/2 \rfloor}^2 \leq {n_i \choose \lfloor n_i/2 \rfloor}
\]
so ${n_i \choose \lfloor n_i/2 \rfloor} \leq 2^r$ contrary to our initial assumption.
We have thus shown that $K_{i,j} = A_{n_i} \times A_{n_j}$.
We can therefore invoke \cref{SubsDirectProdGroups} and conclude that $K = A_{n_1} \times \dots \times A_{n_r}$ as required.
\end{proof}

\begin{cor} \label{LargeProductsForSquareFree}

Let $n_1, \dots, n_r, k$ be integers with $\min\{n_1, \dots, n_r\} > 2k > 0$, and let 
$
g \in S_{n_1} \times \dots \times S_{n_r}
$
such that the lengths of the cycles of the projection of $g$ to $S_{n_i}$ all exceed $k$ for every $1 \leq i \leq r$. Then the subgroups $ S_{n_1-k} \times \dots \times S_{n_r-k}$ and $\langle g \rangle$ invariably generate $S_{n_1} \times \dots \times S_{n_r}$.

\end{cor}

\begin{proof}

Denote by $H$ the subgroup of $S_{n_1} \times \dots \times S_{n_r}$ generated by $ S_{n_1-k} \times \dots \times S_{n_r-k}$ and a conjugate $\rho$ of $g$.
In order to show that $H =  S_{n_1} \times \dots \times S_{n_r}$, we will invoke \cref{SubsDirectProdGroups} whose assumptions we verify next.
The first assumption is satisfied because for every positive integer $n$, the group $S_n$ does not have a nonabelian simple quotient.
We shall now check that $H$ satisfies the three conditions in \cref{SubsDirectProdGroups}.

For the first condition, we need to check the surjectivity of the restriction of the sign homomorphism $S_{n_1} \times \dots \times S_{n_r} \to (\mathbb Z/ 2\mathbb Z)^r$ to $H$.
This follows at once from the surjectivity of the restriction of this homomorphism to $S_{n_1-k} \times \dots \times S_{n_r-k}$, a consequence of the fact that $n_i - k > k \geq 1$ for every $1 \leq i \leq r$.
The second condition is an immediate consequence of \cref{InvariableGnerationProp}.

To check the third condition, we take $1 \leq i < j \leq r$ (with $n_i = n_j$), and denote by $H_{i,j}$ the projection of $H$ to $S_{n_i} \times S_{n_j}$.
Suppose toward a contradiction that $H_{i,j}$ is a proper subgroup of $S_{n_i} \times S_{n_j}$.
In view of the second condition verified above, the subgroup $H_{i,j}$ projects onto both $S_{n_i}$ and $S_{n_j}$ so Goursat's Lemma provides us with a nontrivial group $K$ 
and surjective homomorphisms $\varphi \colon S_{n_i} \to K$, $\psi \colon S_{n_j} \to K$ such that $H_{i,j} = S_{n_i} \times_{K} S_{n_j}$.
Since $K$ is (isomorphic to) a nontrivial quotient of a symmetric group, its abelianization is necessarily nontrivial, so $S_{n_i} \times_{K^\textup{ab}} S_{n_j}$ is a proper subgroup of $S_{n_i} \times S_{n_j}$ that contains both $H_{i,j}$ and the commutator subgroup of $S_{n_i} \times S_{n_j}$. 
It follows that the restriction to $H_{i,j}$ of the natural homomorphism $S_{n_i} \times S_{n_j} \to (S_{n_i} \times S_{n_j})^{\textup{ab}}$ is not surjective.
This contradicts the first condition verified above, namely the surjectivity of the map $H \to (\mathbb Z / 2 \mathbb Z)^r$.
We have thus shown that $H_{i,j} = S_{n_i} \times S_{n_j}$ so the third condition is verified.

We can therefore invoke \cref{SubsDirectProdGroups} and conclude that $H = S_{n_1} \times \dots \times S_{n_r}$ as required.
\end{proof}

\section{Braided Sets}

\begin{defi}

A nonempty set $X$ equipped with a bijection $R \colon X \times X \to X \times X$ is said to be a braided set if
\[
(R \times \mathrm{id}_X) \circ (\mathrm{id}_X \times R) \circ (R \times \mathrm{id}_X) = (\mathrm{id}_X \times R)  \circ (R \times \mathrm{id}_X) \circ (\mathrm{id}_X \times R)
\]
as maps from $X \times X \times X$ to $X \times X \times X$. 

\end{defi}

This relation is sometimes called the set-theoretic Yang--Baxter equation, and the braided set $X$ is sometimes said to be a solution of this equation.

\begin{defi}

We denote by $\pi_1 \colon X \times X \to X$ and $\pi_2 \colon X \times X \to X$ the projections.
In case the function $\pi_2 \circ R$ is a bijection on each fiber of $\pi_2$, and the function $\pi_1 \circ R$ is a bijection on each fiber of $\pi_1$, we say that $X$ is nondegenerate.

\end{defi}

A nondegenerate braided set is sometimes also called a birack.

\begin{defi}

Given braided sets $(X, R)$ and $(X', R')$, we say that a function $f \colon X \to X'$ is a morphism (of braided sets) if $(f \times f)(R(x,y)) = R'(f(x), f(y))$.

\end{defi}

We obtain the category of braided sets. Products exist in this category.

\begin{defi}

We say that a braided set $X$ is trivial if $R(x,y) = (y,x)$ for all $(x,y) \in X \times X$.
A braided set $X$ will be called squarefree if $R(x,x) = (x,x)$ for all $x \in X$.

\end{defi}

For a positive integer $k$, we denote the trivial braided set of cardinality $k$ by $\mathcal T_k$. 
In the sequel, the trivial braided set $\mathcal T_2 = \{0,1\}$ will play an important role.

\begin{defi}

A braided set $(X,R)$ is said to be self-distributive if for all $x,y \in X$ we have $\pi_1(R(x,y)) = y.$
In this case, we use the notation $x^y = \pi_2(R(x,y))$.

\end{defi}

For a self-distributive braided set $X$, and every $y \in X$, the function $x \mapsto x^y$ is a bijection from $X$ to $X$, so a self-distributive braided set is necessarily nondegenerate. The category of self-distributive braided sets is therefore equivalent to the category of racks,
and the subcategory of squarefree self-distributive braided sets is equivalent to the subcategory of quandles.

\begin{defi}

We say that a subset $X_0 \subseteq X$ of a braided set $X$ is a braided subset if $R$ restricts to a bijection from $X_0 \times X_0$ to $X_0 \times X_0$ or equivalently, if $X_0$ is a braided set and the inclusion $X_0 \to X$ is a morphism of braided sets. We write $X_0 \leq X$ to indicate that $X_0$ is a braided subset of $X$. 
Given an indexing set $I$, and elements $x_i \in X$ for every $i \in I$, we denote by
\[
\langle x_i \rangle_{i \in I} = \bigcap_{\substack{X_0 \leq X \\ x_i \in X_0 \text{ for every } i \in I }} X_0
\]
the braided subset of $X$ generated by all the $x_i$ for $i \in I$.

\end{defi}

In case $X$ is a finite rack, this definition agrees with \cref{DefSubRackGen}.

\begin{ex}

Let $G$ be a group, and let $C$ be a union of conjugacy classes of $G$.
We endow $C$ with the structure of a braided set by
\[
R(x,y) = (y,x^y) = (y, y^{-1} x y), \quad x,y \in C.
\] 
This braided set is squarefree, self-distributive, and is trivial if and only if the subgroup of $G$ generated by $C$ is abelian.

\end{ex}

\begin{ex}  \label{Nexample}

We denote by $\mathcal N = \{\eta, \xi\}$ the unique nontrivial self-distributive braided set on two elements. We have $\eta^\eta = \eta^\xi = \xi$ and $\xi^\xi = \xi^\eta = \eta$.
This braided set is not squarefree.

\end{ex}

The category of braided sets admits a (unique) final object - the trivial braided set $\mathcal T_1$. 

\begin{defi}

For a braided set $(X, R)$ we denote by $\tau \colon X \to X_{\textup{triv}}$ the morphism of braided sets characterized by the following universal property. 
The braided set $X_{\textup{triv}}$ is trivial, and for every trivial braided set $Y$, and every morphism $\gamma \colon X \to Y$, there exists a unique morphism $\eta \colon X_{\textup{triv}} \to Y$ such that $\gamma = \eta \circ \tau$. We say that $X_{\textup{triv}}$, or rather $\tau$, is the trivialization of $X$.

\end{defi}

\begin{prop}

A trivialization of a braided set $X$ exists and is unique up to an isomorphism.

\end{prop}

\begin{proof}

Let $\sim$ be the smallest equivalence relation on $X$ such that for all $x,y,z,w \in X$ with \[R(x,y) = (z,w)\] we have $x \sim w$ and $y \sim z$.
We let $X_{\textup{triv}}$ be the set of equivalence classes in $X$ for $\sim$, and denote by $\tau \colon X \to X_{\textup{triv}}$ the map taking an element to its equivalence class.
One readily checks that $\tau$ is a morphism.

To check the universal property let $Y$ be a trivial braided set, and let $\gamma \colon X \to Y$ be a morphism. 
Since the braided sets $X_{\textup{triv}}, Y$ are trivial, and $\tau$ is surjective, we just need to find a function $\eta \colon X_{\textup{triv}} \to Y$ with $\gamma = \eta \circ \tau$.
For that, it suffices to check that for every $x,y \in X$ with $\tau(x) = \tau(y)$ we have $\gamma(x) = \gamma(y)$.
This follows at once from the definition of $\sim$ and the fact that $\gamma$ is a morphism.

Uniqueness up to an isomorphism follows (as usual) from the universal property. 
\end{proof}

We call the equivalence classes appearing in the proof the connected components of $X$.
These connected components are the fibers of the map $\tau \colon X \to X_{\textup{triv}}$.
In case $X$ is a finite rack, this definition agrees with \cref{DefRackComponents}.

\begin{ex}

For a disjoint union $C$ of conjugacy classes $C_1, \dots, C_k$ of a group $G$, such that $C$ generates $G$, the trivialization of $C$ is the map $C \to C_{\textup{triv}} = \{C_1, \dots, C_k\}$ sending each element to the conjugacy class in which it lies.
The connected components of $C$ as a braided set are $C_1, \dots, C_k$.

\end{ex}

\begin{defi}

For a braided set $(X,R)$, a positive integer $n$, and an integer $1 \leq i < n$ we let $\sigma_i \in B_n$ act from the right on the $n$-fold Cartesian product $X^n$ of $X$ by $\mathrm{id}_{X^{i-1}} \times R \times \mathrm{id}_{X^{n-i-1}}$ namely
\[
(x_1, \dots, x_{i-1}, x_i, x_{i+1}, x_{i+2} \dots, x_n)^{\sigma_i} = (x_1, \dots, x_{i-1}, R(x_i, x_{i+1}), x_{i+2} \dots, x_n)
\] 
so we get a right action of $B_n$ on $X^n$.

\end{defi}
 
The association of $X^n$ to $X$ is a functor from the category of braided sets to the category of right actions of $B_n$.
In case $X$ is trivial, this action factors through the permutation action of $S_n$ on $X^n$.
In the special case where $X$ is a group $G$, we recover the action of $B_n$ on $\operatorname{Mor}(F_n, G)$,
and in case $X = C$ is a generating set for $G$ that is stable under conjugation by the elements of $G$, we recover the action of $B_n$ on $\operatorname{Mor}^C(F_n, G)$.

\begin{defi}

We denote the coinvariants of the action of $B_n$ on $X^n$, namely the collection of orbits, by $X^n/{B_n}$ and consider the disjoint union
\[
S_X = \bigcup_{n=1}^\infty X^n/B_n.
\]
Since the action of $B_m \times B_{n-m}$ on $X^m \times X^{n-m}$ is compatible with the natural inclusion $B_m \times B_{n-m} \hookrightarrow B_n$ and the identification $X^m \times X^{n-m} = X^n$, 
the set $S_X$ endowed with the binary operation of concatenation of representatives of orbits is a semigroup. 
We call $S_X$ the semigroup of coinvariants (or the structure semigroup) of the braided set $X$.

\end{defi}

\subsection{The Structure (Semi)group} \label{CPFW}

In case our braided set $X$ is self-distributive, the structure semigroup is given by the presentation
\[
S_X = \langle X : xy = yx^{y} \ \text{for all} \ (x,y) \in X^2 \rangle.
\]

\begin{defi} \label{DefRackGroup}

To a self-distributive braided set $X$ we also functorially associate the group given by the same presentation
\[
\Gamma_X = \langle X : y^{-1}xy = x^y \ \text{for all} \ (x,y) \in X^2 \rangle.
\]
This group is sometimes called the structure group of $X$.

\end{defi}

The morphism of braided sets $\gamma_X \colon X \to \Gamma_X$ enjoys the following universal property.
For every group $G$ and a morphism of braided sets $f \colon X \to G$ there exists a unique homomorphism of groups $\varphi \colon \Gamma_X \to G$ such that $f = \varphi \circ \gamma_X$. At times we will abuse notation viewing elements of $X$ as sitting inside $\Gamma_X$ via $\gamma_X$ (or even inside $S_X$).

We note that the natural homomorphism $\Gamma_X^{\textup{ab}} \to \Gamma_{X_{\textup{triv}}}$ is an isomorphism.
For a trivial braided set, the structure group is a free abelian group on the braided set.

\begin{defi}

We view the group of permutations of $X$ as acting on $X$ from the right.
We then have a homomorphism from $\Gamma_X$ to the group of permutations of $X$ sending $y \in X$ to the permutation $x \mapsto x^y$.
The image of this homomorphism will be denoted by $\operatorname{Inn}(X)$.

\end{defi}

One sometimes calls $\operatorname{Inn}(X)$ the group of inner automorphisms of $X$.
We get a right action of $\Gamma_X$ (or equivalently, of $\operatorname{Inn}(X)$) on $X$ whose orbits are the connected components of $X$.
The kernel of the homomorphism $\Gamma(X) \to \operatorname{Inn}(X)$ is contained in the center of $\Gamma_X$. In particular, in case $X$ is finite, the center of $\Gamma_X$ is of finite index in $\Gamma_X$.
In this case we denote for $x \in X$ by $m_x$ the order of the image of $x$ in $\operatorname{Inn}(X)$.
Then $x^{m_x}$ lies in the center of $S_X$ (and of $\Gamma_X$) and we put
\[
z = \prod_{x \in X} x^{m_x} \in Z(S_X).
\]

\begin{prop} \label{LocalizationExistence}

Let $S$ be a semigroup and let $z \in Z(S)$. Then there exists a monoid $S[z^{-1}]$ and a morphism of semigroups $\varphi \colon S \to S[z^{-1}]$ with $\varphi(z)$ invertible in $S[z^{-1}]$ such that for every monoid $M$ and a semigroup homomorphism $\psi \colon S \to M$ with $\psi(z)$ invertible, there exists a unique homomorphism of monoids $\theta \colon S[z^{-1}] \to M$ satisfying $\psi  = \theta \circ \varphi$.

\end{prop}

\begin{proof}

Consider first the semigroup $\mathbb N \times S$ whose elements we write as $z^{-m}w$ for a nonnegative integer $m$, and $w \in S$. The product is given by
\[
z^{-m}w \cdot z^{-n} v = z^{-(m+n)}wv, \quad m,n \in \mathbb N, \ v,w \in S_X.
\]
There is a natural homomorphism of semigroups $S \to \mathbb N \times S$ sending $w \in S_X$ to $z^{-0}w$.
We define an equivalence relation $\sim$ on $\mathbb N \times S$ by $z^{-m}w \sim z^{-n}v$ if there exists a positive integer $r$ such that 
\[
z^{n+r}w = z^{m+r}v
\]
in $S$. Since multiplication in $\mathbb N \times S$ descends to multiplication on the set of equivalence classes for $\sim$ we can let $S[z^{-1}]$ be the quotient semigroup, and take $\varphi$ to be the composition $S  \to \mathbb N \times S \to S[z^{-1}]$.
We note that $z^{-1}z$ is the identity element of $S[z^{-1}]$ so $S[z^{-1}]$ is indeed a monoid.
The inverse of $\varphi(z)$ in $S[z^{-1}]$ is given by $z^{-2} z$ so $\varphi(z)$ is indeed invertible.
A routine check establishes the required universal property of $S[z^{-1}]$.
\end{proof}

\begin{prop} \label{StructureGroupViaLocalization}

The natural homomorphism of monoids $S_X[z^{-1}] \to \Gamma_X$ is an isomorphism.

\end{prop}

\begin{proof}

We first show that the monoid $S_X[z^{-1}]$ is a group. Since $S_X[z^{-1}]$ is generated by $X$ and $z^{-2}z$ it suffices to show that (the image in $S_X[z^{-1}]$ of) every $x \in X$ is invertible in $S_X[z^{-1}]$. 
Fix an $x \in X$. As $z^{-2}z$ and $y^{m_y}$ lie in the center of $S_X[z^{-1}]$ for every $y \in X$, we have
\[
x \cdot x^{m_x-1} \prod_{y \in X \setminus \{x\}} y^{m_y} \cdot z^{-2}z = 1 = 
x^{m_x-1} \prod_{y \in X \setminus \{x\}} y^{m_y} \cdot z^{-2}z \cdot x
\]
so $x$ is indeed invertible in $S_X[z^{-1}]$. We have thus shown that $S_X[z^{-1}]$ is a group.

Since $S_X[z^{-1}]$ is a group, the natural map $X \to S_X[z^{-1}]$ is a morphism of braided sets.
The universal property of $\Gamma_X$ therefore provides us with a group homomorphism $\Gamma_X \to S_X[z^{-1}]$
which is seen to be the inverse of $S_X[z^{-1}] \to \Gamma_X$ by checking that this is the case on the image of $X$.
\end{proof}

\begin{defi} \label{NotationTrivialization}

Let $X$ be a finite self-distributive braided set, and let $C_1, \dots, C_k$ be the connected components of $X$.
For nonnegative integers $n_1, \dots, n_k$, and $n = n_1 + \dots + n_k$, we put
\[
X(n_1, \dots, n_k) =  \{(x_1, \dots, x_n) \in X^n :  \ |\{1 \leq i \leq n : x_i \in C_j \}| = n_j \text{ for every } 1 \leq j \leq k \},
\]
and consider those tuples that generate $X$ as a braided set, namely we define
\[
X^*(n_1, \dots, n_k) = \{(x_1, \dots, x_n) \in X(n_1, \dots, n_k) :  \langle x_1, \dots, x_n \rangle = X \}.
\]
We define also
$
\left(C_1^{n_1} \times \dots \times C_k^{n_k}\right)^* = 
\left(C_1^{n_1} \times \dots \times C_k^{n_k}\right) \cap X^*(n_1, \dots, n_k).
$

\end{defi}

In case $n_1, \dots, n_k$ are positive, the elements of a tuple in $X(n_1, \dots, n_k)$ generate $X$ as a braided set if and only if they generate the group $\Gamma_X$, or equivalently the group $\operatorname{Inn}(X)$.
We note that $X^*(n_1, \dots, n_k)$ is stable under the action of $B_n$, and that for a nonnegative integer $N$ the set
\[
S_X^*(N) = \bigcup_{n_1, \dots, n_k \geq N} X^*(n_1, \dots, n_k)/B_n
\]
is an ideal (thus also a subsemigroup) of $S_X$. 

\begin{ex} \label{ExampleDefinitionBraidedSetToGroup}

In case $X$ is the disjoint union $C$ of conjugacy classes $C_1, \dots, C_k$ of a finite group $G$ such that $C$ generates $G$, and $n_1, \dots, n_k$ are positive, we have $X^*(n_1, \dots, n_k) = \operatorname{Sur}^C(F_n,G;n_1, \dots, n_k)$.

\end{ex}

\begin{prop} \label{BlockForBraidGroup}

The subset $C_1^{n_1} \times \dots \times C_k^{n_k}$ of $X(n_1, \dots, n_k)$ is a block for $B_n$ and its stabilizer is $B_{n_1, \dots, n_k}$. 

\end{prop}

\begin{proof}

This follows from the fact that $C_1^{n_1} \times \dots \times C_k^{n_k}$ is a fiber of the map 
\[
X(n_1, \dots, n_k) \to X_{\textup{triv}}(n_1, \dots, n_k),
\]
from the fact that this map is a morphism of sets with a right action of $B_n$, 
and from the fact that the action of $B_n$ on $X_{\textup{triv}}(n_1, \dots, n_k)$ factors through the permutation action of $S_n$ on $X_{\textup{triv}}(n_1, \dots, n_k)$.
\end{proof}

The set
$
\left(C_1^{n_1} \times \dots \times C_k^{n_k}\right)_1^* = \{ (x_1, \dots, x_n) \in \left( C_1^{n_1} \times \dots \times C_k^{n_k} \right)^* : x_1 \cdots x_n =1\}
$
is therefore stable under the action of $B_{n_1, \dots, n_k}$.

\begin{cor} \label{ActionOfColoredBraidGroupCor}

The natural map 
$C_1^{n_1} \times \dots \times C_k^{n_k} / B_{n_1, \dots, n_k} \to X(n_1, \dots, n_k)/B_n$ is a bijection.

\end{cor}

\begin{proof}

This follows from \cref{BijectionFromBlocks} and \cref{BlockForBraidGroup} because the orbit under the action of $B_n$ of every element from $X(n_1, \dots, n_k)$ contains an element from $C_1^{n_1} \times \dots \times C_k^{n_k}$.
\end{proof}

In what follows, every time we say that a positive integer is large enough, we mean large enough compared to $|X|$ (and not to any additional quantity).


\begin{prop} \label{MultSurjProp}

For positive integers $n_1, \dots, n_k$ large enough, 
for integers $m_1, \dots, m_k$, the integer $m = m_1 + \dots + m_k$, and $w \in X(m_1, \dots, m_k)/B_m$, the function 
\[
M_w \colon X^*(n_1, \dots,n_k)/B_n \to X^*(n_1+m_1, \dots, n_k + m_k)/B_{n+m}, \quad M_w(v) = wv
\]
of multiplication by $w$ from the left is surjective. 

\end{prop}

\begin{proof}

Since composition of surjective functions is surjective, it suffices to treat the case $m=1$, namely we assume that $w \in X$, so $w \in C_j$ for some $1 \leq j \leq k$.
Take a representative 
\[
s \in X^*(n_1, \dots, n_{j-1}, n_j+1, n_{j+1}, \dots, n_k)
\]
for an orbit under the action of $B_{n+1}$.
As $n_j$ is large enough, it follows from the pigeonhole principle that there exists an $x \in C_j$ that appears in $s$ more than $m_x$ times. We can therefore assume, without changing the orbit of $s$ under the action of $B_{n+1}$, that the first $m_x + 1$ entries of $s$ are all equal to $x$. In particular, we can write $s = x^{m_x} s'$ (the equality taking place in $S_X$) for some 
\[
s' \in X^*(n_1, \dots, n_{j-1}, n_j+1-m_x, n_{j+1}, \dots, n_k)/B_{n+1-m_x}.
\]

Since $x$ and $w$ lie in the same connected component of $X$, there exists $g \in \operatorname{Inn}(X)$ such that $x^g = w$.
The entries of (a representative of) $s'$ generate $X$, so these entries generate $\operatorname{Inn}(X)$ as a group.
In particular, we can write $g$ as a product some of these entries and their inverses (allowing repetition).
Therefore, in order to show that $s = w^{m_x}s'$ in $S_X$ and deduce the required surjectivity, 
it suffices to observe that for every $y \in X$ that appears in (a representative of) $s'$, and every $x_1, \dots, x_\ell \in X$ for which $x_1 \cdots x_\ell \in Z(S_X)$, 
we have the equality
$
x_1 \cdots x_\ell s' = x_1^y \cdots x_\ell^y s'
$ 
in $S_X$.
\end{proof}

\begin{cor} \label{MultInjCor}

For every $N$ large enough, and every $w \in S_X$, multiplication by $w$ from the left is an injective function from $S_X^*(N)$ to itself.

\end{cor}

\begin{proof}

Since composition of injective functions is injective, it suffices to consider the case $w \in C_j$ for some $1 \leq j \leq k$.
In this case, it is enough to show that for integers $n_1, \dots, n_k \geq N$, the function
\[
M_w \colon X^*(n_1, \dots, n_{j-1},n_j,n_{j+1}, \dots, n_k)/B_n \to X^*(n_1, \dots, n_{j-1},n_j+1,n_{j+1}, \dots, n_k)/B_{n+1}
\]
of multiplication by $w$ from the left is injective.
It follows from \cref{MultSurjProp} that for large enough $N$, the function $(n_1, \dots, n_k) \mapsto |X^*(n_1, \dots, n_k)/B_n|$ is nonincreasing in each coordinate hence eventually constant because its values are positive integers.
Therefore, for large enough $N$, the function $M_w$ maps one finite set onto another finite set of the same cardinality.
The required injectivity follows.
\end{proof}

\begin{thm} \label{EmbIntoGroupThm}

For every $N$ large enough, the natural homomorphism of semigroups $S_X^*(N) \to \Gamma_X$ is injective. 

\end{thm}

\begin{proof}

In view of \cref{StructureGroupViaLocalization}, it suffices to show that the natural homomorphism of semigroups $S_X^*(N) \to S_X[z^{-1}]$ is injective.
For that, let $u,v \in S_X^*(N)$ for which $z^r u = z^r v$ for some positive integer $r$. Invoking \cref{MultInjCor} with $w = z^r$, we get that $u = v$ so injectivity follows.
\end{proof}

\subsection{Products of Braided Sets}

\begin{defi}

Let $X,Y$ be self-distributive braided sets.
Then the group homomorphism \[\Gamma_{X \times Y} \to \Gamma_X \times \Gamma_Y\] (arising from the functoriality of the structure group) induces a homomorphism
\[
\Gamma'_{X \times Y} \to \Gamma_X' \times \Gamma_Y'
\]
on the commutator subgroups. In case the latter homomorphism is injective, we say that $(X,Y)$ is a productive pair.

\end{defi}

\begin{prop} \label{TrivialProductivity}

Let $X$ be a self-distributive braided set. Then the pair $(X,\mathcal T_2)$ is productive.

\end{prop}

\begin{proof}

Since $\Gamma_{\mathcal T_2}$ is abelian, we need to show that the natural surjection $\pi \colon \Gamma'_{X \times \mathcal T_2} \to \Gamma_X'$ is an isomorphism.
We will do so by constructing an inverse. The inclusion of the braided subset $X \times \{0\}$ into $X \times \mathcal T_2$ induces a group homomorphism 
$\psi_0 \colon \Gamma_{X} \stackrel{\sim}{\to} \Gamma_{X \times \{0\}} \to \Gamma_{X \times \mathcal T_2}$. We then get a homomorphism $\psi_0' \colon \Gamma_X' \to \Gamma_{X \times \mathcal T_2}'$.
Arguing in the same way with $X \times \{1\}$ in place of $X \times \{0\}$, we get a homomorphism $\psi_1' \colon \Gamma_X' \to \Gamma_{X \times \mathcal T_2}'$.
We claim that these homomorphisms coincide, namely $\psi_0' = \psi_1'$.

The commutator subgroup is generated by conjugates of commutators of elements in a generating set, so it suffices to check that $\psi_0'([x,y]^g) = \psi_1'([x,y]^g)$ for every $x,y \in X$ and $g \in \Gamma_X$. Equivalently, it is enough to show that $\psi_0([x,y])^{\psi_0(g)} = \psi_1([x,y])^{\psi_1(g)}$.
Since conjugation by $\psi_0(g)$ coincides with conjugation by $\psi_1(g)$, our task is to show that $[\psi_0(x), \psi_0(y)] = [\psi_1(x), \psi_1(y)]$. Indeed we have
\begin{equation} \label{ZerosTurningOnes}
\begin{split}
(x,0)(y,0)(x,0)^{-1}(y,0)^{-1} = (x,1)(y,0)(x,1)^{-1}(y,0)^{-1} = (x,1)(y,1)(x,1)^{-1}(y,1)^{-1}.
\end{split}
\end{equation}
We have thus proven our claim that $\psi_0' = \psi_1'$. We henceforth denote this map by $\psi'$.

Now we claim that $\psi'$ is an inverse of the natural surjection $\pi$.
Since the inclusion of the braided set $X \times \{0\}$ into $X \times \mathcal T_2$ is a section of the projection $X \times \mathcal T_2 \to X$, it follows that $\pi \circ \psi'$ is the identity on $\Gamma_X'$. It remains to check that the homomorphism $\psi' \circ \pi$ is the identity on $\Gamma_{X \times \mathcal T_2}'$. 
It suffices to show that $\psi' \circ \pi$ acts as the identity on some generating set of $\Gamma_{X \times \mathcal T_2}'$. Since $X \times \mathcal T_2$ is a generating set for $\Gamma_{X \times \mathcal T_2}$ which is stable under conjugation, the subgroup $\Gamma_{X \times \mathcal T_2}'$ is generated by commutators of elements in $X \times \mathcal T_2$. Therefore, it is enough to check that $\psi' \circ \pi$ acts as the identity on such commutators. 
This follows at once from \cref{ZerosTurningOnes}. We have thus shown that $\psi'$ and $\pi$ are inverse to each other.
\end{proof}

\begin{cor} \label{FirstInjectivityAssocGroup}

For a productive pair $(X,Y)$ of self-distributive braided sets, the natural group homomorphism
$
\Gamma_{X \times Y} \to \Gamma_X \times \Gamma_Y \times \Gamma_{X \times Y}^{\textup{ab}}
$
is injective.

\end{cor}

\begin{proof}

Let $a$ be an element in the kernel of our homomorphism. Since $a$ maps to the identity in $\Gamma_{X \times Y}^{\textup{ab}}$ we get that $a \in \Gamma_{X \times Y}'$.
It follows from the definition of a productive pair that $a = 1$. We have thus shown that the kernel of our homomorphism is trivial, so our homomorphism is indeed injective.
\end{proof}

\begin{defi}

We say that a pair of braided sets $(X,Y)$ is synchronized if the natural surjection 
\[
(X \times Y)_{\text{triv}} \to X_{\text{triv}} \times Y_{\text{triv}}
\] 
is a bijection.

\end{defi}

For every braided set $X$ and every squarefree braided set $Y$, the pair $(X,Y)$ is synchronized.
In particular, for every braided set $X$, the pair $(X, \mathcal T_2)$ is synchronized.

\begin{cor} \label{SecondInjectivityAssocGroup}

For a productive synchronized pair $(X,Y)$ of self-distributive braided sets, the natural group homomorphism
$
\Gamma_{X \times Y} \to \Gamma_X \times \Gamma_Y \times \Gamma_{X_{\textup{triv}} \times Y_{\textup{triv}}}
$
is injective.

\end{cor}

\begin{proof}

In view of \cref{FirstInjectivityAssocGroup}, it suffices to check that the natural homomorphism from $\Gamma_{X \times Y}^{\textup{ab}}$, which we identify with $\Gamma_{(X\times Y)_{\textup{triv}}}$, to $\Gamma_{X_{\textup{triv}} \times Y_{\textup{triv}}}$ is injective.
This follows at once from our assumption that the pair $(X,Y)$ is synchronized.
\end{proof}


\begin{thm} \label{ActionOnPairsBraidGroup}

Let $(X,Y)$ be a productive synchronized pair of self-distributive braided sets. 
Let $N$ be sufficiently large, and let $v,w \in S^*_{X \times Y}(N)$ satisfying the following three conditions.
\begin{itemize}

\item The projection of $v$ to $S_X$ coincides with the projection of $w$ to $S_X$.

\item The projection of $v$ to $S_Y$ coincides with the projection of $w$ to $S_Y$.

\item For every $t \in X_{\textup{triv}} \times Y_{\textup{triv}}$ the number of entries of $v$ that project to $t$ coincides with the number of entries of $w$ that project to $t$.

\end{itemize}
Then $v=w$.

\end{thm}

\begin{proof}

By \cref{EmbIntoGroupThm}, it suffices to show that $v = w$ in $\Gamma_{X \times Y}$.
In view of \cref{SecondInjectivityAssocGroup}, it is enough to show that the images of $v$ and $w$ in each of the groups $\Gamma_X$, $\Gamma_Y$, $\Gamma_{X_{\textup{triv}} \times Y_{\textup{triv}}}$ agree. This is guaranteed by the three conditions that $v$ and $w$ satisfy.
\end{proof}

\subsection{Image of Stabilizer in $S_n$}

As in \cref{NotationTrivialization}, let $X$ be a finite self-distributive braided set, let $C_1, \dots, C_k$ be the connected components of $X$, let $n$ be a positive integer, and let $s = (x_1, \dots, x_n) \in X^n$.  
For $1 \leq j \leq k$ put $D_j = \{1 \leq i \leq n : x_i \in C_j\}$ and $n_j = |D_j|$. We will make the identification 
\[
\{\sigma \in S_n : \sigma(D_j) = D_j \text{ for every } 1 \leq j \leq k \} = S_{n_1} \times \dots \times S_{n_k}.
\]
Recall that $B_{n_1, \dots, n_k}$ is the inverse image of $S_{n_1} \times \dots \times S_{n_k}$ under the homomorphism from $B_n$ to $S_n$.

\begin{cor} \label{RestrictionOnImage}

The image in $S_n$ of the stabilizer in $B_n$ of $s$ is contained in $S_{n_1} \times \dots \times S_{n_k}$. 
In other words, the stabilizer of $s$ in $B_n$ is the stabilizer of $s$ in $B_{n_1, \dots, n_k}$.

\end{cor}

\begin{proof}

This is an immediate consequence of \cref{BlockForBraidGroup}.
\end{proof}

%

\begin{thm} \label{OptConj}

Suppose that $n_1, \dots, n_k$ are sufficiently large, and that $\langle x_1, \dots, x_n \rangle = X$.
Then the image in $S_{n_1} \times \dots \times S_{n_k}$ of the stabilizer in $B_n$ (equivalently, in $B_{n_1, \dots, n_k}$) of $s$ contains $A_{n_1} \times \dots \times A_{n_k}$.

\end{thm}

\begin{proof}

We invoke \cref{HomogeneityProducts} so given $X_j,Y_j \subseteq D_j$ with $|X_j| = |Y_j| = \lfloor n_j/2 \rfloor$, we need to show that there exists $g \in B_n$ such that $s^g = s$ and $X_j^g = Y_j$ for all $1 \leq j \leq k$.
For $1 \leq i \leq n$ put
\[
u_i = 
\begin{cases}
1 &i \in X_j \text{ for some } 1 \leq j \leq k \\
0 &\text{otherwise}
\end{cases} \quad
v_i = 
\begin{cases}
1 &i \in Y_j \text{ for some } 1 \leq j \leq k \\
0 &\text{otherwise}.
\end{cases}
\]
Viewing $((x_1, u_1), \dots, (x_n, u_n))$ and $((x_1, v_1), \dots, (x_n, v_n))$ as elements of $(X \times \mathcal T_2)^n$,
our task is to show that $((x_1, u_1), \dots, (x_n, u_n))^g = ((x_1, v_1), \dots, (x_n, v_n))$ for some $g \in B_n$.

Since the entries of $s$ generate $X$, they also generate $\operatorname{Inn}(X)$. As $\mathcal T_2$ is a trivial braided set, it follows that the entries of $((x_1, u_1), \dots, (x_n, u_n))$ generate $\operatorname{Inn}(X \times \mathcal T_2)$, so these entries also generate $X \times \mathcal T_2$ because they map surjectively onto $(X \times \mathcal T_2)_{\text{triv}} = X_{\text{triv}} \times \mathcal T_2$. Similarly, the entries of $((x_1, v_1), \dots, (x_n, v_n))$ generate $X \times \mathcal T_2$.
Therefore, we need to show that $((x_1, u_1), \dots, (x_n, u_n)) = ((x_1, v_1), \dots, (x_n, v_n))$ as elements in $S^*_{X \times \mathcal T_2}(\lfloor \min\{n_1, \dots, n_k\}/2 \rfloor)$.

By \cref{TrivialProductivity}, the pair $(X, \mathcal T_2)$ is productive, so we can resort to \cref{ActionOnPairsBraidGroup} once we check the three conditions therein.
The first condition is satisfied because both $((x_1, u_1), \dots, (x_n, u_n))$ and $((x_1, v_1), \dots, (x_n, v_n))$ project to $s$ in $S_X$.
The second condition is satisfied because by assumption
\[
|X_1| + \dots + |X_k| = |Y_1| + \dots + |Y_k|
\]
so there exists $g \in B_n$ such that $(u_1, \dots, u_n)^g = (v_1, \dots, v_n)$ as elements in $\mathcal T_2^n$.
To see that the third condition is satisfied, we fix $1 \leq j \leq k$ and $\lambda \in \mathcal T_2$. We have
\[
|\{1 \leq i \leq n : x_i \in C_j, \ u_i = \lambda\}| = 
\begin{cases}
|X_j| &\lambda = 1\\
n_j - |X_j| &\lambda = 0
\end{cases}
\]
and similarly
\[
|\{1 \leq i \leq n : x_i \in C_j, \ v_i = \lambda\}| = 
\begin{cases}
|Y_j| &\lambda = 1\\
n_j - |Y_j| &\lambda = 0.
\end{cases}
\]
Since $|X_j| = |Y_j|$ by assumption, the third condition in \cref{ActionOnPairsBraidGroup} is indeed satisfied.
\end{proof}

\begin{remark} \label{MoreGeneralThanRacks}

It is likely possible to extend \cref{OptConj} to nondegenerate (but not necessarily self-distributive) finite braided sets. 
To do this, one associates to a nondegenerate finite braided set $Y$ its derived self-distributive braided set $X$, as in \cite{Sol}.
The key point is that this association gives an isomorphism $Y^n \to X^n$ of $B_n$-sets.
We do not know whether it is possible to extend \cref{OptConj} to (some family of) degenerate braided sets.
Perhaps a first step would be to obtain a version of the results in \cref{CPFW} for more general braided sets.

\end{remark}

\begin{lem} \label{ProduceTranspstnImage}

Suppose that $X$ is squarefree, and let $1 \leq \alpha < \beta \leq n$ with $x_\alpha = x_\beta$.
Then the image in $S_n$ of the stabilizer of $s$ in $B_n$ contains the transposition $(\alpha \ \beta)$.

\end{lem}

\begin{proof}

Since $X$ is squarefree and self-distributive, we see that the element 
\[
R_{\alpha, \beta} = \sigma_{\beta-1} \sigma_{\beta-2} \cdots \sigma_{\alpha+1} \sigma_{\alpha} \sigma_{\alpha+1}^{-1} \cdots \sigma_{\beta-2}^{-1} \sigma_{\beta-1}^{-1} \in B_n
\]
lies in the stabilizer of $(x_1, \dots, x_n)$, and that its image in $S_{n}$ is the transposition $(\alpha \ \beta)$.
\end{proof}

\begin{prop} \label{SquarfreeTranspositionProp}

Suppose that $X$ is squarefree, set $N = \max_{1 \leq j \leq k} |C_j|$, and assume that $n_j > N$ for every $1 \leq j \leq k$.
Then the image in $S_{n_1} \times \dots \times S_{n_k}$ of the stabilizer in $B_{n_1, \dots, n_k}$ of $s$ surjects onto $(\mathbb Z / 2 \mathbb Z)^k$ under the sign homomorphisms.

\end{prop}

\begin{proof}


Fix $1 \leq j \leq k$. Our task is to find an element in the stabilizer of $(x_1, \dots, x_n)$ whose image in $S_{n_j}$ is an odd permutation, for instance a transposition,
and whose image in $S_{n_r}$ for every $1 \leq r \leq k$ with $r \neq j$ is trivial.
It follows from our definition of $N$ and the assumption on $n_j$ that there exist two indices $\alpha < \beta$ in $D_j$ for which $x_\alpha = x_\beta$.
The required element is supplied to us by \cref{ProduceTranspstnImage}.
\end{proof}


The argument in the proof above also gives the following.

\begin{cor} \label{NotContainedAlternatingGroup}

Suppose that $X$ is squarefree, and that $n_j > |C_j|$ for some $1 \leq j \leq k$.
Then the image in $S_n$ of the stabilizer in $B_n$ (equivalently, in $B_{n_1, \dots, n_k}$) of $s$ is not contained in $A_n$. 

\end{cor}

%
%
%
%

\begin{cor} \label{SqfreeCorSn}

Suppose that $n_1, \dots, n_k$ are sufficiently large, that $\langle x_1, \dots, x_n \rangle = X$, and that $X$ is squarefree.
Then the image in $S_{n}$ of the stabilizer in $B_n$ (equivalently, in $B_{n_1, \dots, n_k}$) of $s$ is $S_{n_1} \times \dots \times S_{n_k}$.

\end{cor}

\begin{proof}

Denote this image by $H$. 
By \cref{OptConj}, $H$ contains $A_{n_1} \times \dots \times A_{n_k}$.
It follows from \cref{SquarfreeTranspositionProp} that the restriction of the quotient map $S_{n_1} \times \dots \times S_{n_k} \to S_{n_1} \times \dots \times S_{n_k}/A_{n_1} \times \dots \times A_{n_k}$ to $H$ is surjective. We conclude that $H = S_{n_1} \times \dots \times S_{n_k}$ as required.
\end{proof}

From \cref{OptConj} and \cref{SqfreeCorSn} we get \cref{ConnectedLargeMonodromy}.

\begin{ex} \label{TwoElementExample}

By considering the case of the nontrivial self-distributive braided set $X = \mathcal N = \{\eta, \xi\}$ from \cref{Nexample}, we show that the squarefreeness assumption in \cref{SqfreeCorSn} is necessary in order to obtain a result stronger than in \cref{OptConj}.

First note that in this case $X$ is connected.
Now take any $s \in X^n$, and observe that its entries necessarily generate $X$. 
It is readily checked that for each of the generators $\sigma_i$ for $1 \leq i \leq n-1$,
the parity of the number of appearances of $\eta$ in $s$ differs from the parity of the number of appearances of $\eta$ in $s^{\sigma_i}$.
As a result, for every $g \in B_n$, the number of times $\eta$ appears in $s$ is congruent mod $2$ to the number of times $\eta$ appears in $s^g$ if and only if the image of $g$ in $S_n$ lies in $A_n$. 
We conclude that the image in $S_n$ of the stabilizer of $s$ in $B_n$ is contained in $A_n$.

\end{ex}

%
%
%
%

We sketch an additional proof of \cref{SqfreeCorSn}.

\begin{prop} \label{CoincidencesProp}

There exists a nonnegative integer $r$ for which the following holds. 
If $n_1, \dots, n_k$ are large enough, and $\langle x_1, \dots, x_n \rangle = X$, then the 
$B_{n_1, \dots, n_k}$-orbit of $s$ contains both
\begin{itemize}

\item an $n$-tuple whose entries indexed by the first $n_j-r$ indices in $D_j$ coincide for every $1 \leq j \leq k$;

\item and an $n$-tuple $(y_1, \dots, y_n)$ with $|\{1 \leq i \leq n : y_i = x\}| > r$ for every $x \in X$.

\end{itemize}

\end{prop}

\begin{proof}

This follows from \cref{MultSurjProp} and \cref{ActionOfColoredBraidGroupCor}.
\end{proof}

\begin{proof}[Proof of \cref{SqfreeCorSn}]

It follows from the first item in \cref{CoincidencesProp} in conjunction with \cref{ProduceTranspstnImage} that there exists a nonnegative integer $r$ such that the image in $S_{n_1} \times \dots \times S_{n_k}$ of the stabilizer in $B_n$ of $s$ contains the subgroup $S_{n_1 - r} \times \dots \times S_{n_k - r}$.
From the second item in \cref{CoincidencesProp} and \cref{ProduceTranspstnImage} we conclude that this image also contains a conjugate of an element whose projections to $S_{n_1}, \dots, S_{n_r}$ have all of their cycles of lengths exceeding $r$. We conclude by invoking \cref{LargeProductsForSquareFree}.
\end{proof}

\subsubsection{Results for Small $n$}

\begin{prop} \label{SmSmalln}

Let $m \geq 2$ be an integer, and let $X$ be the conjugacy class of all transpositions in the group $S_m$.
Then for a positive integer $n$ and every $s \in X^n$ whose entries generate $X$, the restriction to the stabilizer of $s$ in $B_n$ of the homomorphism to $S_n$ is a surjection. 

\end{prop}

\begin{remark}

If there exists an $n$-tuple $s \in X^n$ whose entries generate $X$, then $n \geq m-1$.

\end{remark}

\begin{proof}

Let $s = ((i_1 \ j_1), \dots, (i_n \  j_n))$ be an $n$-tuple of transpositions in $S_m$ that generates $X$.
Consider the simple graph $\Lambda$ whose set of vertices is $\{1, \dots, n\}$ with $1 \leq \alpha < \beta \leq n$ adjacent in case 
\[
\{i_\alpha, j_\alpha \} \cap \{i_\beta, j_\beta \} \neq \emptyset.
\]
We claim that $\Lambda$ is connected.

To prove the claim, suppose toward a contradiction that there exist nonempty disjoint subsets $I,J$ of $\{1, \dots, n\}$ with $I \cup J = \{1, \dots, n\}$ such that there is no edge between any index in $I$ and any index in $J$. From our definition of $\Lambda$ it follows that the nonempty subsets
\[
\mathcal I = \bigcup_{\alpha \in I} \{i_\alpha, j_\alpha\}, \quad \mathcal J = \bigcup_{\beta \in J} \{i_\beta, j_\beta \}
\]
of $\{1, \dots, m\}$ are disjoint.
We conclude that
\[
X = \langle (i_1 \ j_1), \dots, (i_n \  j_n) \rangle \leq \{(i \ j) : i\neq j, \ i,j \in \mathcal I \text{ or } i,j \in \mathcal J \} \lneq X.
\]
This contradiction concludes the proof of our claim that $\Lambda$ is connected.

Denote by $H$ the image in $S_n$ of the stabilizer in $B_n$ of $s$. 
We note that for every $1\leq \alpha < \beta \leq n$ that are adjacent in $\Lambda$, the stabilizer of $s$ in $B_n$ contains the element
\[
R_{\alpha, \beta}^3 = \sigma_{\beta-1} \sigma_{\beta-2} \cdots \sigma_{\alpha+1} \sigma_{\alpha}^3 \sigma_{\alpha+1}^{-1} \cdots \sigma_{\beta-2}^{-1} \sigma_{\beta-1}^{-1}.
\]
This element maps to the transposition $(\alpha \ \beta)$ in $S_n$, so $(\alpha \ \beta) \in H$. We consider the subgroup
\[
H_0 = \left \langle \{ (\alpha \ \beta) : 1\leq \alpha < \beta \leq n \text{ are adjacent in } \Lambda \} \right \rangle
\]
of $H$ generated by all such transpositions. Since $\Lambda$ is connected, it follows that $H_0$ is a transitive subgroup of $S_n$.
The only transitive subgroup of $S_n$ that is generated by transpositions is $S_n$ itself, so $H_0 = S_n$ and thus $H = S_n$ as required.
\end{proof}

Let $C \subseteq G$ be a generating set of a group $G$ which is a disjoint union of conjugacy classes of $G$.
We have a left action of $G$ on the sets 
$
\operatorname{Mor}(F_n,G), \ \operatorname{Mor}^C(F_n,G), \ \operatorname{Sur}^C(F_n,G), \ \operatorname{Sur}^C_1(F_n,G)$
by postcomposition with conjugation. This action commutes with the right action of $B_n$, so we get a right action of $B_n$ on the sets of orbits 
$
G \backslash \operatorname{Mor}(F_n,G), \ G \backslash \operatorname{Mor}^C(F_n,G), \ G \backslash \operatorname{Sur}^C(F_n,G), \ G \backslash \operatorname{Sur}^C_1(F_n,G). 
$
With our identification of $\operatorname{Mor}(F_n, G)$ with $G^n$, this left action of $G$ on  is given by
\[
{}^g(g_1, \dots, g_n) = (gg_1g^{-1}, \dots, g g_n g^{-1}), \quad (g_1, \dots, g_n) \in G^n, \ g \in G.
\]

\begin{prop} \label{CyclicShiftProp}

For a positive integer $n$, and every $(g_1,  g_2, \dots, g_{n-1}, g_n) \in G \backslash \operatorname{Mor}(F_n, G)$ we have
\[
(g_1, g_2, \dots, g_{n-1}, g_n)^{\sigma_{n-1} \sigma_{n-2} \cdots \sigma_2 \sigma_1} = (g_n, g_1, g_2, \dots, g_{n-1})
\]
as classes in $G \backslash \operatorname{Mor}(F_n,G)$.
\end{prop}

\begin{proof}

Viewing our representatives in $G \backslash \operatorname{Mor}(F_n,G)$ as elements in $\operatorname{Mor}(F_n,G)$, we see that the action of $\sigma_{n-1} \sigma_{n-2} \cdots \sigma_2 \sigma_1 \in B_n$ is given by
$
(g_1, g_2, \dots, g_{n-1}, g_n)^{\sigma_{n-1} \sigma_{n-2} \cdots \sigma_2 \sigma_1} = (g_n, g_1^{g_n}, g_2^{g_n}, \dots, g_{n-1}^{g_n}).  
$
The required equality of classes in $G \backslash \operatorname{Mor}(F_n,G)$ is then seen by conjugating the right hand side by $g_n$.
\end{proof}

\begin{defi}

Let $G$ be a group and let $C \subset G$ be a conjugacy class. We say that $C$ is abundant if for some (equivalently, every) $x \in C$ there exists $y \in G$ such that the set $\{x^{y^r} : r \in \mathbb Z\}$ (of conjugates of $x$ by elements of the cyclic subgroup of $G$ generated by $y$) generates $G$.

\end{defi}

\begin{cor} \label{AbundanceThm}

Let $G$ be a finite simple group. Then there exists an abundant conjugacy class $C \subset G$.

\end{cor}

\begin{proof}

This is a special case of \cite[Corollary 4 (ii)]{BGH}.
\end{proof}

\begin{prop} \label{AbundanceImpliesSn}

Let $G$ be a finite group, and let $C \subset G$ be an abundant conjugacy class.
Then for every positive integer $n$ that is divisible by $|G|^2$, there exists $s \in G \backslash \operatorname{Sur}^C_1(F_n,G)$ such that the image of the stabilizer of $s$ under the homomorphism from $B_n$ to $S_n$ contains an $n$-cycle.

\end{prop}

\begin{proof}

Since $G$ is a finite group, and $C$ is an abundant conjugacy class, there exist $x \in C$ and $y \in G$ such that
$\langle x^{y^r} : 0 \leq r \leq |G| - 1 \rangle = G.$ We set
\[
g = x \cdot x^y \cdot x^{y^2} \cdots x^{y^{|G|-1}} \in G, \quad s_0 = (x, x^y, x^{y^2}, x^{y^{|G|-1}}) \in G^{|G|}
\]
and denote by $s$ the $\frac{n}{|G|}$-fold concatenation of $s_0$ with itself. Since $n$ is divisible by $|G|^2$, we see that $\frac{n}{|G|}$ is a multiple of $|G|$, so multiplying the entries of $s$ (in order) gives $1 \in G$ because $g^{|G|} = 1$. 
We conclude that $s$ represents an element of $\operatorname{Sur}^C_1(F_n,G)$.

It follows from \cref{CyclicShiftProp} that the class of $s$ in $G \backslash \operatorname{Sur}^C_1(F_n,G)$ is mapped under $\sigma_{n-1} \cdots \sigma_1$ to a class represented by a one-step cyclic right shift of $s$. 
Since conjugating this representative by $y^{-1} \in G$ gives $s$, we conclude that $\sigma_{n-1} \cdots \sigma_1$ lies in the stabilizer of the class of $s$ in $G \backslash \operatorname{Sur}^C_1(F_n,G)$. 
As $\sigma_{n-1} \cdots \sigma_1$ maps to an $n$-cylce in $S_n$, the image in $S_n$ of the stabilizer of $s$ contains an $n$-cycle.
\end{proof}

\begin{cor} \label{FiniteSimpleGroup}

Let $G$ be a finite simple group.
Then for every positive integer $n$ that is divisible by $|G|^2$ there exists $s \in G \backslash \operatorname{Sur}_1(F_n,G)$ such that the image of the stabilizer of $s$ under the homomorphism from $B_n$ to $S_n$ contains an $n$-cycle.

\end{cor}

\begin{proof}

This is an immediate consequence of \cref{AbundanceThm} and \cref{AbundanceImpliesSn}.
\end{proof}

\section{Proof of \cref{MobiusVonMangoldtConjectureLargeFF}} \label{THEproof}

For brevity of notation, we set $k = d_\lhd(G)$.

\subsection{The M\"obius Function}

Our task here is to prove the first part of \cref{MobiusVonMangoldtConjectureLargeFF}, namely that
\[
\sum_{K \in \mathcal E_q^C(G;n_1, \dots, n_{k})} (-1)^{|\operatorname{ram}(K)|} =  o\left(\left|\mathcal E_q^C(G;n_1, \dots, n_{k})\right|\right),\quad q \to \infty, \quad \gcd(q, |G|) = 1,
\]
assuming $n_j > |C_j|$ for some $1 \leq j \leq k$. 

%

 
We recall from \cite[Section 11.4, Theorem 11.1, Lemma 11.2, Proposition 11.4]{LWZB} that there exists a smooth separated scheme $\mathsf{Hur}_{G,C}^{n_1, \dots, n_k}$ of finite type over $\mathbb Z[|G|^{-1}]$ with
 \[
 \mathsf{Hur}_{G,C}^{n_1, \dots, n_k}(\F_q) = \mathcal E_q^C(G;n_1, \dots, n_k).
 \]
 The scheme $\mathsf{Hur}_{G,C}^{n_1, \dots, n_k}$ is (pure) of relative dimension $n$ over $\mathbb Z[|G|^{-1}]$.
 

Let $\textup{Conf}^{n_1, \dots, n_k}$ be the $k$-colored configuration space of $n_j$ unordered points of color $j$, for every $1 \leq j \leq k$, on the affine line such that all points are distinct (whether they have the same color or not).
We can view a point of this space over a field as a $k$-tuple of pairwise coprime monic squarefree polynomials of degrees $n_1, \dots, n_k$ over that field.
The space $\textup{Conf}^{n_1, \dots, n_k}$ is a smooth separated scheme of finite type over $\mathbb Z$.
\cite[Proposition 11.4]{LWZB} provides us with a finite \'etale map 
\[
\pi \colon \mathsf{Hur}_{G,C}^{n_1, \dots, n_k} \to \textup{Conf}^{n_1, \dots, n_k}
\] 
such that $\pi(K) = (D_K(1), \dots, D_K(k))$ on the level of $\F_q$-points.

Let $\textup{PConf}^n$ be the locus in $\mathbb A^n_{\mathbb Z}$ where all coordinates are pairwise distinct.
We can view a point of this space over a field as an (ordered) $n$-tuple of distinct scalars from that field.
We consider the finite \'etale map $\rho \colon \textup{PConf}^n \to \textup{Conf}^{n_1, \dots, n_k}$ which given an $n$-tuple $(\lambda_1, \dots, \lambda_n)$ of distinct scalars from a field, assigns the color $j$ to the scalars $\lambda_{n_1 + \dots + n_{j-1} + 1}, \dots, \lambda_{n_1 + \dots + n_j}$ for every $1 \leq j \leq k$. We can also write
\[
\rho(\lambda_1, \dots, \lambda_n) =  \left( \prod_{r=1}^{n_j} (T-\lambda_{n_1 + \dots + n_{j-1} + r}) \right)_{j = 1, \dots, k}
\]
where the right hand side is a $k$-tuple of pairwise coprime monic squarefree polynomials.
The map $\rho$ is a Galois cover with Galois group $S_{n_1} \times \dots \times S_{n_k}$ that acts by permuting the roots of each polynomial. 

Take an auxiliary prime number $\ell > n$ with $\ell$ not dividing $q$. We fix an isomorphism of fields $\iota \colon \overline{\mathbb Q_\ell} \to \mathbb C$, and will at times (silently) identify these fields via $\iota$.

Using $\rho$ we can view every finite-dimensional representation $W$ of $S_{n_1} \times \dots \times S_{n_k}$ over $\overline{\mathbb Q_\ell}$ as a lisse (\'etale) $\overline{\mathbb Q_\ell}$-sheaf on $\textup{Conf}^{n_1, \dots, n_k}$ punctually pure of weight $0$. We write $\chi_W$ for the character of $W$, and for $(f_1, \dots, f_k) \in \textup{Conf}^{n_1, \dots, n_k}(\F_q)$ we denote by $(\sigma_{f_1}, \dots, \sigma_{f_k})$ the conjugacy class in $S_{n_1} \times \dots \times S_{n_k}$ corresponding to the permutation induced by $\operatorname{Frob}_q$ on (the roots of) $f_j$ for every $1 \leq j \leq k$. The cycle structure of $\sigma_{f_j}$ is the multiset of degrees of the monic irreducible factors of $f_j$. We therefore have
\[
\operatorname{tr}\left(\operatorname{Frob}_q, W_{(f_1, \dots, f_k)}\right) = \chi_W(\sigma_{f_1}, \dots, \sigma_{f_k})
\]
an expression for the trace of Frobenius on the stalk of (the sheaf corresponding to) $W$ at a geometric point of $\textup{Conf}^{n_1, \dots, n_k}$ lying over $(f_1, \dots, f_k)$.

Let $\pi^* W$ be the lisse sheaf (punctually pure of weight $0$) on $\mathsf{Hur}_{G,C}^{n_1, \dots, n_k}$ obtained by pulling back $W$. 
For every $K \in \mathcal E_q^C(G;n_1, \dots, n_k)$ we have
\[
\operatorname{tr}\left(\operatorname{Frob}_q, (\pi^* W)_{K}\right) = \operatorname{tr}\left(\operatorname{Frob}_q, W_{\pi(K)}\right) = \operatorname{tr}\left(\operatorname{Frob}_q, W_{(D_K(1), \dots, D_K(k))}\right) = \chi_W\left(\sigma_{D_K(1)}, \dots, \sigma_{D_K(k)}\right).
\]

In the special case $W = \textup{sgn}_1 \boxtimes \dots \boxtimes \textup{sgn}_k$ the sign representation of $S_{n_1} \times \dots \times S_{n_k}$,
for $K$ in $\mathcal E_q^C(G;n_1, \dots, n_k)$ we have
\[
\chi_{\textup{sgn}_1 \boxtimes \dots \boxtimes \textup{sgn}_k}\left(\sigma_{D_K(1)}, \dots, \sigma_{D_K(k)}\right) = (-1)^n \cdot (-1)^{|\operatorname{ram}(K)|}
\]
so
\[
\sum_{K \in \mathcal E_q^C(G;n_1, \dots, n_k)} (-1)^{|\operatorname{ram}(K)|}  = 
(-1)^n \sum_{K \in  \mathsf{Hur}_{G,C}^{n_1, \dots, n_k}(\F_q)} \operatorname{tr}\left(\operatorname{Frob}_q, (\pi^* \textup{sgn}_1 \boxtimes \dots \boxtimes \textup{sgn}_k)_{K}\right).
\]


We can assume that $\mathsf{Hur}_{G,C}^{n_1, \dots, n_k}$ has an $\F_q$-point since otherwise the sum above is over the empty set, so the statement to be proven holds trivially. It follows that $\mathsf{Hur}_{G,C}^{n_1, \dots, n_k}$ has an $\F_q$-rational component, so from the Lang--Weil bound applied to that component we get that
\[
\liminf_{\substack{q \to \infty \\ \gcd(q, |G|)  = 1}} \frac{|\mathcal E_q^C(G;n_1, \dots, n_k)|}{q^n} > 0.
\]
Our task is therefore to show that
\[
\lim_{\substack{q \to \infty \\ \gcd(q, |G|)  = 1}} 
\frac{1}{q^n} \sum_{K \in \mathsf{Hur}_{G,C}^{n_1, \dots, n_k}(\F_q) } \operatorname{tr}(\operatorname{Frob}_q, (\pi^* W)_{K}) = 0, \quad W = \textup{sgn}_1 \boxtimes \dots \boxtimes \textup{sgn}_k.
\]
From now until almost the end of the proof, we will task ourselves with computing (under suitable assumptions on $n_1, \dots, n_k$) the limit above for an arbitrary finite-dimensional representation $W$ of $S_{n_1} \times \dots \times S_{n_k}$ over $\overline{\mathbb Q_\ell}$. 
It is only at the very end that we will specialize again to $W = \textup{sgn}_1 \boxtimes \dots \boxtimes \textup{sgn}_k$ and deduce that the limit is indeed $0$ in case $n_j > |C_j|$ for some $1 \leq j \leq k$.


By the Grothendieck--Lefschetz trace formula, we have
\begin{equation*}
\begin{split}
&\lim_{\substack{q \to \infty \\ \gcd(q, |G|)  = 1}} 
\frac{1}{q^n} \sum_{K \in \mathsf{Hur}_{G,C}^{n_1, \dots, n_k}(\F_q) } \operatorname{tr}(\operatorname{Frob}_q, (\pi^* W)_{K}) = \\
&\lim_{\substack{q \to \infty \\ \gcd(q, |G|)  = 1}} 
\frac{1}{q^n} \sum_{i=0}^{2n} (-1)^i \operatorname{tr}(\operatorname{Frob}_q, H_c^i( \mathsf{Hur}_{G,C}^{n_1, \dots, n_k} \times_{\mathbb Z[|G|^{-1}]} \overline{\F_q} , \pi^*W))
\end{split}
\end{equation*}
where we use the notation $\pi^* W$ also for the pullback of this sheaf to $\overline{\F_q}$.
We claim first that there is no contribution to the limit from $0 \leq i \leq 2n-1$.
Indeed since $\pi^* W$ is punctually pure of weight $0$, Deligne's Riemann Hypothesis gives an upper bound of $q^{i/2}$ on the absolute value of each eigenvalue of $\operatorname{Frob}_q$ on $H_c^i( \mathsf{Hur}_{G,C}^{n_1, \dots, n_k} \times_{\mathbb Z[|G|^{-1}]} \overline{\F_q} , \pi^*W)$, and the dimension over $\overline{\mathbb Q_\ell}$ of these cohomology groups is bounded independently of $q$, so dividing the trace of $\operatorname{Frob}_q$ by $q^n$ and taking $q \to \infty$ gives $0$ in the limit. 
We therefore have
\begin{equation*}
\begin{split}
&\lim_{\substack{q \to \infty \\ \gcd(q, |G|)  = 1}} 
\frac{1}{q^n} \sum_{i=0}^{2n} (-1)^i \operatorname{tr}(\operatorname{Frob}_q, H_c^i( \mathsf{Hur}_{G,C}^{n_1, \dots, n_k} \times_{\mathbb Z[|G|^{-1}]} \overline{\F_q} , \pi^*W)) =  \\
&\lim_{\substack{q \to \infty \\ \gcd(q, |G|)  = 1}} \frac{1}{q^n} \operatorname{tr}(\operatorname{Frob}_q, H_c^{2n}( \mathsf{Hur}_{G,C}^{n_1, \dots, n_k} \times_{\mathbb Z[|G|^{-1}]} \overline{\F_q} , \pi^*W)).
\end{split}
\end{equation*}

Since representations of finite groups in characteristic $0$ are semisimple, we can find a subrepresentation $U$ of $W$ for which 
\[
W = U \oplus W^{S_{n_1} \times \dots \times S_{n_k}}.
\]
We then have $U^{S_{n_1} \times \dots \times S_{n_k}} = 0$, and
\begin{equation*}
\begin{split}
&\lim_{\substack{q \to \infty \\ \gcd(q, |G|)  = 1}} \frac{1}{q^n} \operatorname{tr}(\operatorname{Frob}_q, H_c^{2n}( \mathsf{Hur}_{G,C}^{n_1, \dots, n_k} \times_{\mathbb Z[|G|^{-1}]} \overline{\F_q} , \pi^*W)) = \\
&\lim_{\substack{q \to \infty \\ \gcd(q, |G|)  = 1}} \frac{1}{q^n} \operatorname{tr}(\operatorname{Frob}_q, H_c^{2n}( \mathsf{Hur}_{G,C}^{n_1, \dots, n_k} \times_{\mathbb Z[|G|^{-1}]} \overline{\F_q} , \pi^*U)) + \\ &\dim_{\overline{\mathbb Q_\ell}} W^{S_{n_1} \times \dots \times S_{n_k}} \lim_{\substack{q \to \infty \\ \gcd(q, |G|)  = 1}} \frac{1}{q^n} \operatorname{tr}(\operatorname{Frob}_q, H_c^{2n}( \mathsf{Hur}_{G,C}^{n_1, \dots, n_k} \times_{\mathbb Z[|G|^{-1}]} \overline{\F_q} , \overline{\mathbb Q_\ell})). 
\end{split}
\end{equation*}
Since $\mathsf{Hur}_{G,C}^{n_1, \dots, n_k} \times_{\mathbb Z[|G|^{-1}]} \overline{\F_q}$ is of dimension $n$, the action of $\operatorname{Frob}_q$ on its topmost compactly supported \'etale cohomology (with constant coefficients, namely $ \overline{\mathbb Q_\ell}$) is via multiplication by $q^n$, so the above equals
\begin{equation*}
\begin{split}
&\lim_{\substack{q \to \infty \\ \gcd(q, |G|)  = 1}} \frac{1}{q^n} \operatorname{tr}(\operatorname{Frob}_q, H_c^{2n}( \mathsf{Hur}_{G,C}^{n_1, \dots, n_k} \times_{\mathbb Z[|G|^{-1}]} \overline{\F_q} , \pi^*U)) + \\ 
&\dim_{\overline{\mathbb Q_\ell}} W^{S_{n_1} \times \dots \times S_{n_k}} \cdot \dim_{\overline{\mathbb Q_\ell}} H_c^{2n}( \mathsf{Hur}_{G,C}^{n_1, \dots, n_k} \times_{\mathbb Z[|G|^{-1}]} \overline{\F_q} , \overline{\mathbb Q_\ell}). 
\end{split}
\end{equation*}


Since $\pi$ is finite, it follows from the Leray spectral sequence with compact supports that
\[
H_c^{2n}( \mathsf{Hur}_{G,C}^{n_1, \dots, n_k} \times_{\mathbb Z[|G|^{-1}]} \overline{\F_q} , \pi^*U) \cong 
H_c^{2n}( \textup{Conf}^{n_1, \dots, n_k} \times_{\mathbb Z} \overline{\F_q} , \pi_* \pi^*U) 
\]
where $\pi_*$ is the pushforward of lisse sheaves by $\pi$.
One readily checks that $\pi_* \pi^* U \cong \pi_* \overline{\mathbb Q_\ell} \otimes U$ so 
\[
H_c^{2n}( \textup{Conf}^{n_1, \dots, n_k} \times_{\mathbb Z} \overline{\F_q} , \pi_* \pi^*U)  \cong 
H_c^{2n}( \textup{Conf}^{n_1, \dots, n_k} \times_{\mathbb Z} \overline{\F_q} , \pi_* \overline{\mathbb Q_\ell} \otimes U).
\]
Since the sheaves $\pi_* \overline{\mathbb Q_\ell} $ and $U$ are self-dual, from Poincar\'e duality we get that
\[
H_c^{2n}( \textup{Conf}^{n_1, \dots, n_k} \times_{\mathbb Z} \overline{\F_q} ,  \pi_* \overline{\mathbb Q_\ell} \otimes U) \cong 
H^{0}( \textup{Conf}^{n_1, \dots, n_k} \times_{\mathbb Z} \overline{\F_q} ,  \pi_* \overline{\mathbb Q_\ell} \otimes U).
\]

It follows from the proof of \cite[Lemma 10.3]{LWZB} that with our choice of $\ell$ we have
\[
H^{0}( \textup{Conf}^{n_1, \dots, n_k} \times_{\mathbb Z} \overline{\F_q} ,  \pi_* \overline{\mathbb Q_\ell} \otimes U) = 
H^{0}( \textup{Conf}^{n_1, \dots, n_k}(\mathbb C)  ,  \pi_* \overline{\mathbb Q_\ell} \otimes U)
\]
where on the right hand side we take singular cohomology, viewing $\pi_* \overline{\mathbb Q_\ell} \otimes U$ as a local system on $\textup{Conf}^{n_1, \dots, n_k}(\mathbb C)$,
or rather as a representation of the fundamental group $B_{n_1, \dots, n_k}$ of $\textup{Conf}^{n_1, \dots, n_k}(\mathbb C)$.
Therefore we have
\[
H^{0}( \textup{Conf}^{n_1, \dots, n_k}(\mathbb C)  ,  \pi_* \overline{\mathbb Q_\ell} \otimes U) \cong (\pi_* \overline{\mathbb Q_\ell} \otimes U)^{B_{n_1, \dots, n_k}}.
\]

Denote by $\overline{\mathbb Q_\ell}\left(C_1^{n_1} \times \dots \times C_k^{n_k}\right)_1^*$ the permutation representation over $\overline{\mathbb Q_\ell}$ associated to the action of $B_{n_1, \dots, n_k}$ on $\left(C_1^{n_1} \times \dots \times C_k^{n_k}\right)_1^*$.
It follows from the proof of \cite[Theorem 12.4]{LWZB} that
\[
\pi_*\overline{\mathbb Q_\ell} \cong \overline{\mathbb Q_\ell}\left(C_1^{n_1} \times \dots \times C_k^{n_k}\right)_1^*
\]
so
\[
(\pi_* \overline{\mathbb Q_\ell} \otimes U)^{B_{n_1, \dots, n_k}} \cong \left(\overline{\mathbb Q_\ell}\left(C_1^{n_1} \times \dots \times C_k^{n_k}\right)_1^* \otimes U \right)^{B_{n_1, \dots, n_k}}.
\]
Since permutation representations are self-dual, we have 
\[
\left(\overline{\mathbb Q_\ell}\left(C_1^{n_1} \times \dots \times C_k^{n_k}\right)_1^* \otimes U \right)^{B_{n_1, \dots, n_k}} \cong 
\operatorname{Hom}_{B_{n_1, \dots, n_k}}\left(\overline{\mathbb Q_\ell}\left(C_1^{n_1} \times \dots \times C_k^{n_k}\right)_1^*, U \right).
\]

Let $S$ be a set of representatives for the orbits of the action of $B_{n_1, \dots, n_k}$ on $\left(C_1^{n_1} \times \dots \times C_k^{n_k}\right)_1^*$.
That is, for every $t \in \left(C_1^{n_1} \times \dots \times C_k^{n_k}\right)_1^*$ there exists a unique $s \in S$ that lies in the orbit of $t$ under the action of $B_{n_1, \dots, n_k}$.
For $s \in S$, denoting by $B_{n_1, \dots, n_k,s}$ the stabilizer of $s$ in $B_{n_1, \dots, n_k}$, we see that
\[
\overline{\mathbb Q_\ell}\left(C_1^{n_1} \times \dots \times C_k^{n_k}\right)_1^* \cong 
\bigoplus_{s \in S} \operatorname{Ind}_{B_{n_1, \dots, n_k, s}}^{B_{n_1, \dots, n_k}}\overline{\mathbb Q_\ell}
\]
and as a result
\[
\operatorname{Hom}_{B_{n_1, \dots, n_k}}\left(\overline{\mathbb Q_\ell}\left(C_1^{n_1} \times \dots \times C_k^{n_k}\right)_1^*, U \right) \cong
\bigoplus_{s \in S} \operatorname{Hom}_{B_{n_1, \dots, n_k}}\left( \operatorname{Ind}_{B_{n_1, \dots, n_k, s}}^{B_{n_1, \dots, n_k}}\overline{\mathbb Q_\ell}, U \right).
\]


It follows from Frobenius reciprocity that
\[
\bigoplus_{s \in S} \operatorname{Hom}_{B_{n_1, \dots, n_k}}\left(\operatorname{Ind}_{B_{n_1, \dots, n_k, s}}^{B_{n_1, \dots, n_k}}\overline{\mathbb Q_\ell}, U \right) \cong
\bigoplus_{s \in S} \operatorname{Hom}_{B_{n_1, \dots, n_k, s}}\left(\overline{\mathbb Q_\ell},U \right).
\]
The homomorphism from $B_{n_1, \dots, n_k}$ to $S_{n_1} \times \dots \times S_{n_k}$ arising from the cover $ \textup{PConf}^n(\mathbb C) \to \textup{Conf}^{n_1, \dots, n_k}(\mathbb C)$ is the one we have considered in previous sections.
In case $n_1, \dots, n_k$ are large enough, \cref{SqfreeCorSn} tells us that for every $s \in S$ the restriction to $B_{n_1, \dots, n_k, s}$ of the homomorphism from $B_{n_1, \dots, n_k}$ to $S_{n_1} \times \dots \times S_{n_k}$ is surjective, so
\[
\bigoplus_{s \in S} \operatorname{Hom}_{B_{n_1, \dots, n_k, s}}\left(\overline{\mathbb Q_\ell},U \right) \cong
\bigoplus_{s \in S} \operatorname{Hom}_{S_{n_1} \times \dots \times S_{n_k}}\left(\overline{\mathbb Q_\ell},U \right) \cong 
\bigoplus_{s \in S} U^{S_{n_1} \times \dots \times S_{n_k}} = 0.
\]
Therefore in case $n_1, \dots, n_k$ are large enough, the limit we wanted to compute is
\[
\dim_{\overline{\mathbb Q_\ell}} W^{S_{n_1} \times \dots \times S_{n_k}} \cdot \dim_{\overline{\mathbb Q_\ell}} H^{0}( \mathsf{Hur}_{G,C}^{n_1, \dots, n_k}(\mathbb C), \overline{\mathbb Q_\ell}). 
\]

Let us specialize now to the case $W =  \textup{sgn}_1 \boxtimes \dots \boxtimes \textup{sgn}_k$. In this case we have
\[
U = W = \textup{sgn}_1 \boxtimes \dots \boxtimes \textup{sgn}_k, \quad W^{S_{n_1} \times \dots \times S_{n_k}} = 0.
\]
By assumption $n_j > |C_j|$ for some $1 \leq j \leq k$ so it follows from \cref{NotContainedAlternatingGroup} that 
\[
\bigoplus_{s \in S} \operatorname{Hom}_{B_{n_1, \dots, n_k, s}}\left(\overline{\mathbb Q_\ell},U \right) \cong
\bigoplus_{s \in S} U^{B_{n_1, \dots, n_k, s}} = 0.
\]
Hence in this case the limit we wanted to compute is indeed $0$.

\subsection{The von Mangoldt Function}
Here we prove the second part of \cref{MobiusVonMangoldtConjectureLargeFF}, namely that
\[
\sum_{K \in \mathcal E_q^C(G;n_1, \dots, n_{k})} \mathbf{1}_{|\operatorname{ram}(K)| = k} \sim 
\frac{\left|\mathcal E_q^C(G;n_1, \dots, n_{k})\right|}{n_1 \cdots n_{k}}, \quad q \to \infty, \quad \gcd(q, |G|) = 1,
\]
assuming that $n_1, \dots, n_k$ are sufficiently large.
We will freely use arguments and conclusions from the previous subsection.



We denote by $\Lambda_j$ the von Mangoldt function, which for our purposes is defined on monic squarefree polynomials of degree $n_j$ over $\F_q$, taking the value $n_j$ on irreducible polynomials and the value $0$ on reducible polynomials.
For $K \in \mathcal E_q^C(G;n_1, \dots, n_{k})$ we therefore have 
\[
\mathbf{1}_{|\operatorname{ram}(K)| = k} = \frac{1}{n_1 \cdots n_k} \cdot \prod_{j=1}^k \Lambda_j(D_K(j)).
\]

We denote by $\operatorname{std}_j$ the standard representation of $S_{n_j}$ (of dimension $n_j-1$) over $\overline{\mathbb Q_\ell}$.
By \cite[Lemma 3.6]{SawIntervals}, for every monic squarefree polynomial $f$ of degree $n_j$ over $\F_q$ we have
\[
\Lambda_j(f) = \sum_{i=0}^{n_j-1} (-1)^i \chi_{\wedge^i(\operatorname{std}_j)}(\sigma_f)
\]
where $\sigma_f$ is the conjugacy class in $S_{n_j}$ of the permutation induced by the map $z \mapsto z^q$ on the (necessarily distinct) roots of $f$.
We conclude that for every $K \in \mathcal E_q^C(G;n_1, \dots, n_{k})$ we have
\[
\mathbf{1}_{|\operatorname{ram}(K)| = k} = \frac{1}{n_1 \cdots n_k} \cdot \prod_{j=1}^k \sum_{i=0}^{n_j-1} (-1)^i \chi_{\wedge^i(\operatorname{std}_j)}\left(\sigma_{D_K(j)}\right).
\]

The contribution to the right hand side above from taking the $i=0$ term for every $1 \leq j \leq k$ is $\frac{1}{n_1 \cdots n_k}$, so summing this over all $K \in \mathcal E_q^C(G;n_1, \dots, n_{k})$ gives the required main term $\frac{\left| \mathcal E_q^C(G;n_1, \dots, n_{k})\right|}{n_1 \cdots n_k}$.
Our task is therefore to show that the contribution of any other term is $o\left(\left|\mathcal E_q^C(G;n_1, \dots, n_{k})\right|\right)$.
That is, taking $0 \leq i_1 < n_1, \ \dots, \ 0 \leq i_k < n_k$ not all zero, it suffices to show that
\[
\sum_{K \in \mathcal E_q^C(G;n_1, \dots, n_{k})} \prod_{j=1}^k \chi_{\wedge^{i_j}(\operatorname{std}_j)}\left(\sigma_{D_K(j)}\right) = 
o\left(\left|\mathcal E_q^C(G;n_1, \dots, n_{k})\right|\right), \quad q \to \infty, \quad \gcd(q, |G|) = 1.
\]

We consider the representation
\[
W = \wedge^{i_1}(\operatorname{std}_1) \boxtimes \dots \boxtimes \wedge^{i_k}(\operatorname{std}_k)
\]
of $S_{n_1} \times \dots \times S_{n_k}$ over $\overline{\mathbb Q_\ell}$.
Since $\wedge^{i_j}(\operatorname{std}_j)$ is an irreducible finite-dimensional representation of $S_{n_j}$ over $\overline{\mathbb Q_\ell}$ for every $1 \leq j \leq k$, 
it follows that $W$ is an irreducible finite-dimensional representation of $S_{n_1} \times \dots \times S_{n_k}$.
Our assumption that $i_j > 0$ for some $1 \leq j \leq k$ implies that 
\[
\dim_{\overline{\mathbb Q_\ell}} W = \prod_{j=1}^{k} \dim_{\overline{\mathbb Q_\ell}} \wedge^{i_j}(\operatorname{std}_j) > 1
\] 
so $W^{S_{n_1} \times \dots \times S_{n_k}} = 0$ in view of irreducibility.
In the notation of the previous subsection we therefore have $U = W$.

The sum in which we need to obtain cancellation can be rewritten as
\[
\sum_{K \in \mathcal E_q^C(G;n_1, \dots, n_{k})} \prod_{j=1}^k \chi_{\wedge^{i_j}(\operatorname{std}_j)}\left(\sigma_{D_K(j)}\right) = 
\sum_{K \in \mathcal E_q^C(G;n_1, \dots, n_{k})} \chi_W\left(\sigma_{D_K(1)}, \dots, \sigma_{D_K(k)}\right).
\]
As in the previous subsection we have
\[
\sum_{K \in \mathcal E_q^C(G;n_1, \dots, n_{k})} \chi_W\left(\sigma_{D_K(1)}, \dots, \sigma_{D_K(k)}\right) = 
\sum_{K \in \mathcal E_q^C(G;n_1, \dots, n_{k})} \operatorname{tr}\left(\operatorname{Frob}_q, (\pi^* W)_{K}\right).
\]
By assumption $n_1, \dots, n_k$ are large enough, so from the previous subsection we see that the sum above is indeed $o\left(\left| \mathcal E_q^C(G;n_1, \dots, n_{k}) \right|\right)$.

\section*{Acknowledgments}

I am deeply indebted to Be'eri Greenfeld for his support, interest, and input. 
In particular, I am thankful to him for his guidance on racks, for helpful discussion on the second proof of \cref{SqfreeCorSn}, for informing me of \cite{BGH}, and for proving extensions of \cref{SmSmalln} to other cycle structures.

\end{document}